\documentclass[a4paper, 12pt, draft]{article}
\usepackage[dvips]{graphicx}

\usepackage{amsfonts, amsmath, amsthm}
\usepackage{latexsym}

\theoremstyle{plain}
\newtheorem{thm}{Theorem}[section]
\newtheorem{prop}{Proposition}[section]
\newtheorem{lem}{Lemma}[section]
\newtheorem{cor}{Corollary}[section]

\theoremstyle{definition}

\theoremstyle{remark}
\newtheorem{rem}{Remark}[section]
\newtheorem{ex}{Example}[section]

\numberwithin{equation}{section}

\newcommand{\Z}{\mathbb{Z}}

\newcommand{\R}{\mathbb{R}}
\newcommand{\C}{\mathbb{C}}
\newcommand{\Sph}{\mathbb{S}}

\newcommand{\pa}{\partial}
\newcommand{\eps}{\varepsilon}
\newcommand{\jb}[1]{\langle #1 \rangle}

\DeclareMathOperator{\supp}{\rm supp}
\newcommand{\dal}{\Box}

\newcommand{\dis}{\displaystyle}

\newcommand{\Sum}{\sideset{}{'}\sum}  
\newcommand{\ZP}{\Z_{+}}

\begin{document}
\title{
 Energy decay for systems of semilinear wave equations 
 with dissipative structure in two space dimensions 
 }

\author{
    Soichiro Katayama 
     \thanks{
      Department of Mathematics, Wakayama University. 
      930 Sakaedani, Wakayama 640-8510, Japan. 
      (E-mail: {\tt katayama@center.wakayama-u.ac.jp})
     }
    \and 
    Akitaka Matsumura 
     \thanks{
      Department of Pure and Applied Mathematics, 
      Graduate School of Information Science and Technology, 
      Osaka University.
      Toyonaka, Osaka 560-0043, Japan. 
      (E-mail: {\tt akitaka@ist.osaka-u.ac.jp})
     }
    \and  
    Hideaki Sunagawa
     \thanks{
       Department of Mathematics, Graduate School of Science, 
       Osaka University. 
       Toyonaka, Osaka 560-0043, Japan. 
       (E-mail: {\tt sunagawa@math.sci.osaka-u.ac.jp})
     }
}

\date{\today}   
\maketitle

\vspace{-4mm}\begin{center}
  Dedicated to Professor Shuichi Kawashima 
  on the occasion of his sixtieth birthday\\
\end{center}\vspace{2mm}

\noindent{\bf Abstract:}\ 
We consider the Cauchy problem for systems of semilinear 
wave equations in two space dimensions. We present a structural 
condition on the nonlinearity under which the energy decreases to 
zero as time tends to infinity if the Cauchy data are sufficiently 
small, smooth and compactly-supported.  
\\

\\

\\



\section{Introduction and the main result}

We consider the Cauchy problem for a system of semilinear wave equations
in two space dimensions:
\begin{align}
 \dal u=F(\pa u), 
  \qquad (t,x)\in (0,\infty)\times \R^2,
  \label{eq}
\end{align}
with 
\begin{align}
 u(0,x)=\eps f(x),\ (\pa_t u)(0,x)=\eps g(x), 
  \qquad 
  x\in \R^2,
\label{data}
\end{align}
where 
$u=(u_j)_{1\le j \le N}$ is an $\R^N$-valued unknown function
of $(t,x)\in [0,\infty) \times \R^2$, 
$\Box=\pa_t^2-\Delta_x=\pa_t^2-\pa_{1}^2-\pa_{2}^2$,
and $\pa u=(\pa_a u_j)_{0\le a \le 2, 1\le j \le N}$ 
with the notation
$$
\pa_0=\pa_t=\frac{\pa}{\pa t},\ \ 
\pa_1=\frac{\pa}{\pa x_1},\ \ 
\pa_2=\frac{\pa}{\pa x_2}.
$$
For simplicity, \scalebox{0.95}[1]{we always suppose that 
$f=(f_j)_{1\le j \le  N}$ and $g =(g_j)_{1\le j \le  N}$ 
belong to $C_0^{\infty}(\R^2; \R^N)$}, and that $\eps$ is 
positive and sufficiently small. 
The nonlinear term  
$F(\pa u)=\bigl(F_j(\pa u)\bigr)_{1\le j \le N}$ is 
assumed to be a quadratic smooth function around $\pa u=0$: 
To be more precise, we assume that 
$F \in C^\infty(\R^{3N};\R^N)$ and 
$$
F(\pa u)=F^{\rm q}(\pa u)+F^{\rm c}(\pa u)+O(|\pa u|^4)
$$
in a neighborhood of $\pa u=0$, where the quadratic nonlinear term 
$F^{\rm q}(\pa u)=\bigl(F_j^{\rm q}(\pa u)\bigr)_{1\le j\le N}$ 
and the cubic nonlinear term 
$F^{\rm c}(\pa u)=\bigl(F_j^{\rm c}(\pa u)\bigr)_{1\le j\le N}$ 
are given by
\begin{align*}
 F_j^{\rm q}(\pa u)
 = &
 \sum_{k,l=1}^N \sum_{a,b=0}^2 B_{jkl}^{ab} (\pa_a u_k)(\pa_b u_l),
\\
 F_j^{\rm c}(\pa u)
 = & \sum_{k,l,m=1}^{N}\sum_{a,b,c=0}^{2}  C_{jklm}^{abc} 
 (\pa_a u_k)(\pa_b u_l)(\pa_c u_m)
\end{align*}
with some real constants $B_{jkl}^{ab}$ and $C_{jklm}^{abc}$.
In order to state our conditions, we define the {\it reduced nonlinearity}
$$
F^{\rm q, red}(\omega, Y)
 =\bigl(F_j^{\rm q, red}(\omega, Y)\bigr)_{1\le j \le N}
\text{ and }
F^{\rm c, red}(\omega, Y)
=\bigl(F_j^{\rm c, red}(\omega, Y)\bigr)_{1\le j\le N}
$$
by
\begin{align}
 F_j^{\rm q, red}(\omega, Y)
 = & \sum_{k,l=1}^N\sum_{a,b=0}^2 B_{jkl}^{ab}\omega_a\omega_b Y_k Y_l,
 \nonumber
 \\
 F_j^{\rm c, red}(\omega, Y)
 = & \sum_{k,l,m=1}^{N}\sum_{a,b,c=0}^{2}  C_{jklm}^{abc} 
 \omega_a \omega_b \omega_c Y_k Y_l Y_m
 \label{ReducedNonlinearityC}
\end{align}
for $Y=(Y_j)_{1\le j \le N} \in \R^N$ 
and 
$\omega=(\omega_1,\omega_2) \in \Sph^1$, with the convention $\omega_0=-1$.

It is known that the Cauchy problem \eqref{eq}--\eqref{data} admits a unique 
global solution for small initial data if $F(\pa u)=O(|\pa u|^4)$ near 
$\pa u=0$; however this is not true when we consider general cubic or 
quadratic nonlinearity. Thus the cubic nonlinearity is critical for small data 
global existence in two space dimensions, and we need some structural 
restriction on quadratic and cubic parts of $F$ to obtain global solutions 
for small initial data. 
Alinhac~\cite{Ali01a} proved that if the {\it null condition} 
(or the {\it quadratic null condition})
\begin{equation}
 \label{null2}
 F^{\rm q, red} (\omega, Y)=0,\quad  
 (\omega, Y) \in \Sph^1 \times \R^N
\end{equation}
and the {\it cubic null condition}
\begin{equation}
 \label{null3}
 F^{\rm c, red} (\omega, Y)=0,\quad  
 (\omega, Y) \in \Sph^1 \times \R^N
\end{equation}
are satisfied, then the Cauchy problem \eqref{eq}--\eqref{data} admits a 
unique global solution for small $\eps$. 
More precisely, only the case of single quasi-linear equations is 
treated in \cite{Ali01a}, but we can easily adopt the method in \cite{Ali01a} 
to the system \eqref{eq} (see \cite{Hos06a}). 
It is also known that the null condition \eqref{null2} without \eqref{null3} 
implies the so-called almost global existence for \eqref{eq}--\eqref{data}.
For the related results concerning the quadratic and cubic null conditions in 
two space dimensions, we refer the readers to \cite{Ali01b}, \cite{Go}, 
\cite{Hos95}, \cite{Hos-Kub00}, \cite{Hos-Kub05}, \cite{Kat93} and 
\cite{Kat95}. The (quadratic) null condition was originally introduced by 
Klainerman~\cite{Kla86} as a sufficient condition for small data global 
existence in three space dimensions (see also Christodoulou~\cite{Chr}); 
the cubic null condition is not needed then, 
because the critical nonlinearity is quadratic in three space dimensions.

Concerning single wave equations with cubic nonlinearity in two space
dimensions, Agemi \cite{Age96} introduced another structural condition 
being weaker than the cubic null condition, and conjectured that the small data 
global existence would  follow from his condition. Let $N=1$ and 
$F^{\rm q}(\pa u)\equiv 0$ for a while. Then $F^{\rm c, red}$ has the form 
$F^{\rm c, red}(\omega, Y)=P(\omega) Y^3$ with a polynomial $P$ of cubic order. Agemi's condition is:
\begin{equation}
P(\omega)\ge 0,\quad \omega\in \Sph^1.
\label{Agemi01}
\end{equation}
Observe that \eqref{Agemi01} is equivalent to
\begin{equation}
 YF^{\rm c, red}(\omega, Y)\ge 0, \quad 
 (\omega,Y) \in \Sph^1 \times \R, 
\label{Agemi02}
\end{equation}
and that the cubic null condition \eqref{null3} implies \eqref{Agemi02}.
The Agemi conjecture was solved 
independently by Hoshiga~\cite{Hos} and Kubo~\cite{Ku}: 
Namely, for \eqref{eq}--\eqref{data} with $N=1$ and 
$F^{\rm q}(\pa u)\equiv 0$, it was proved 
that \eqref{Agemi01} implies global existence of solutions for small $\eps$.
For example, the Agemi condition \eqref{Agemi01} is satisfied for the 
nonlinearity $F(\pa u)=-(\pa_t u)^3$, but the cubic null condition is violated 
for this nonlinearity. 
Asymptotic behavior of global solutions under the Agemi condition 
\eqref{Agemi01} was studied in \cite{Ku} and improved in \cite{KMuS}.  
In particular, it was proved in \cite{KMuS} that 
the energy of the global solution $u$ decreases to zero as $t \to \infty$ 
if the inequality in \eqref{Agemi01} is strict, i.e., 
\begin{equation}
 \label{KMuSCond}
 P(\omega)>0,\quad \omega\in \Sph^1. 
\end{equation}
In other words, $F$ satisfying $\eqref{KMuSCond}$ serves as 
a nonlinear dissipation (at least for small data). A typical example 
satisfying \eqref{KMuSCond} is $F(\pa u)=-(\pa_t u)^3$, for which 
$P(\omega)=1$. It should be emphasized that the general theory of nonlinear 
dissipation in Mochizuki--Motai~\cite{MM} does not cover the case of 
$\dal u=-(\pa_t u)^3$ in two space dimensions (see also \cite{TY} 
and the references cited therein for the theory of nonlinear dissipation). 

In this paper, we will unify two global existence results mentioned above:
One is the global existence result under the
quadratic and cubic null conditions \eqref{null2}--\eqref{null3} in 
\cite{Ali01a}; another is the result under the Agemi condition \eqref{Agemi01} 
and $F^{\rm q}(\pa u)\equiv 0$ in \cite{Hos} and \cite{Ku}. 
We will also investigate a condition, corresponding to \eqref{KMuSCond}, 
to ensure that the nonlinearity works as nonlinear dissipation for systems.

Now we would like to introduce our condition:
\begin{enumerate}
\item[(Ag)] There is an $N\times N$-matrix valued continuous function 
${\mathcal A}={\mathcal A}(\omega)$ on $\Sph^1$ such that 
${\mathcal A}(\omega)$ is a positive-definite symmetric matrix for each 
$\omega\in \Sph^1$, and that
$$
 Y \cdot {\mathcal A}(\omega) F^{\rm c, red}(\omega, Y)\ge 0,
 \quad (\omega, Y)\in \Sph^1 \times \R^N,
$$
where the symbol $\,\cdot\,$ denotes the standard inner product in 
$\R^N$.
\end{enumerate}

Here and in what follows, $\R^N$-vectors are always regarded as column vectors.
Observe that the cubic null condition \eqref{null3} implies (Ag) with 
${\mathcal A}(\omega)=I_N$, where $I_N$ is the $N\times N$ identity matrix. 
Observe also that (Ag) with $N=1$ coincides with the Agemi condition 
\eqref{Agemi01}, because we have \eqref{Agemi02}, and ${\mathcal A}(\omega)$ 
in (Ag) plays no essential role when $N=1$. 
Thus we may say that the condition (Ag) is the Agemi condition for systems. 

\begin{thm}  \label{Thm_SDGE}
Suppose that the quadratic null condition \eqref{null2} 
and the condition {\rm (Ag)} are satisfied. 
Then, for any $f, g \in C_0^{\infty}(\R^2;\R^N)$, 
there exists a positive constant $\eps_0$ such that 
the Cauchy problem \eqref{eq}--\eqref{data} admits a unique global $C^{\infty}$-solution
$u$ for $(t,x)\in [0,\infty) \times \R^2$ 
if $\eps \in (0,\eps_0]$. 
\end{thm}

\begin{rem}
For systems \eqref{eq} with cubic nonlinearity, another kind
of extension of the cubic null condition is studied in \cite{Ka}. There is 
no inclusion between the condition in \cite{Ka} and the condition (Ag) here.
\end{rem}
In \cite{KMaS}, systems of semilinear wave equations with
quadratic nonlinearity in three space dimensions are studied, and a sufficient
condition, which is weaker than the null condition, for small data global 
existence is obtained. Our condition (Ag) above can be viewed as a two space 
dimensional version of the condition in \cite{KMaS}. Theorem~\ref{Thm_SDGE} 
can be proved by a method similar to \cite{KMaS} 
(and also to \cite{KMuS}).
However, we need some modification to treat the quadratic nonlinearity 
by using a generalized energy estimate due to 
Alinhac~\cite{Ali01a} and \cite{Ali04} (see Lemma~\ref{GhostEnergy} below). 
Theorem~\ref{Thm_SDGE} will be proved in Section~\ref{PSDGE}.

Now we turn our attention to the decay of the energy. We define the
{\it energy norm} $\|u(t)\|_E$ by
$$
\|u(t)\|_E=\left(\frac{1}{2}\int_{\R^2} \sum_{a=0}^2 |\pa_a u(t,x)|^2 dx \right)^{1/2}.
$$

\begin{thm}  \label{Thm_EnergyDecay}
In addition to \eqref{null2} and {\rm (Ag)}, we assume that 
\begin{equation}
 Y \cdot {\mathcal A}(\omega) F^{\rm c, red}(\omega, Y)\ne  0, 
 \quad (\omega, Y)\in \Sph^1 \times (\R^N\backslash \{0\}).
\label{AgD}
\end{equation}
Let $u$ be the global solution to \eqref{eq}--\eqref{data} whose 
existence is guaranteed by Theorem~\ref{Thm_SDGE}. 
For any $\delta>0$,     
there exists a positive constant $C$ such that 
$$
\|u(t)\|_E\le \frac{C\eps}{(1+ \eps^2\log (t+2))^{\frac{1}{4}-\delta}}, \quad t\ge 0,
$$
provided that $\eps$ is sufficiently small.
\end{thm}

If $N=1$ and $F^{\rm q}(\pa u)\equiv 0$, the assumption in 
Theorem~\ref{Thm_EnergyDecay} is nothing but \eqref{KMuSCond}. 
Although the expression is slightly different, we can easily check that
the energy decay rate in \cite{KMuS} coincides with the above rate. 
In \cite{KMuS}, the decay of the energy was obtained as a corollary to
a general theorem on the pointwise asymptotics of the global solutions
under the condition \eqref{Agemi01}. 
Explicit solvability of some related ordinary differential equations 
plays an essential role in the derivation of the asymptotics in \cite{KMuS}.
It seems quite difficult to apply this method to our system, because
we need to solve a related {\it system} of ODEs explicitly.
Therefore we take another approach to analyze solutions to the related system 
of ODEs without solving it explicitly. 
Theorem~\ref{Thm_EnergyDecay} will be proved in Section~\ref{PED}.

\begin{rem}
Under the assumption of Theorem~\ref{Thm_EnergyDecay}, we also have an
enhanced pointwise decay estimate for $\pa u$. See \eqref{DecayDamping02} below.
\end{rem}

\begin{rem}
In \cite{KMuS}, single but complex-valued wave equations with
cubic gauge-invariant semilinear terms are treated in fact, 
and the complex version of \eqref{Agemi01} was considered. 
However, 
the results on global existence and the energy decay
in \cite{KMuS} are easily recovered by our results here, by rewriting a single 
equation of a complex-valued unknown  as a 
two-component system of real unknowns through the standard identification 
of $\C$ with $\R^2$.
\end{rem}

\begin{rem}
For closely related results on nonlinear Schr\"odinger equations and 
nonlinear Klein-Gordon equations, see \cite{KLS} and \cite{KimS}, respectively.
\end{rem}

We conclude this section by giving some examples satisfying our conditions
\eqref{null2}, (Ag) and \eqref{AgD}.
Throughout this paper, 
we will use the following convention on implicit constants: 
The expression $f=\sum_{\lambda \in \Lambda}' g_{\lambda}$ means that there 
exists a family $\{C_{\lambda}\}_{\lambda \in \Lambda}$ of constants such 
that $f=\sum_{\lambda \in \Lambda} C_{\lambda} g_{\lambda}$.

\begin{ex} [Quadratic terms satisfying the null condition \eqref{null2}]
\label{ex01}
It is well known that the null condition \eqref{null2} is satisfied if and 
only if
\begin{equation}
\label{Null2Express}
F^{\rm q}_j(\pa u)=\Sum_{1\le k,l \le N} Q_0(u_k,u_l)
+ 
\Sum_{\substack{1\le k,l \le N \\ 0\le a,b\le 2}} Q_{ab}(u_k, u_l),
\end{equation}
where the {\it null forms} $Q_0$ and $Q_{ab}$ are given by 
\begin{align}
 Q_0(\phi, \psi)
 =& (\pa_t \phi)(\pa_t \psi)-(\nabla_x \phi) \cdot (\nabla_x \psi),
 \label{nullform_0}\\
 Q_{ab}(\phi, \psi)
 =& (\pa_a\phi)(\pa_b \psi)-(\pa_b \phi)(\pa_a \psi), 
 \qquad 0\le a, b\le 2 
 \label{nullform_ab}
\end{align}
(see \cite{Kla86} for instance). 
Similarly to \eqref{Null2Express}, it is also known that
 the cubic null condition \eqref{null3} is satisfied if and only if
$$
F^{\rm c}_j(\pa u)
=
 \Sum_{\substack{1\le k,l,m \le N\\ 0\le c\le 2 }} 
  (\pa_c u_m)Q_0(u_k,u_l)
 +
 \Sum_{\substack{1\le k,l,m \le N\\ 0\le a,b,c \le 2 }} 
  (\pa_c u_m)Q_{ab}(u_k, u_l)
$$
(see \cite{Kat93} for example). These nonlinear terms can be added freely without
affecting the conditions (Ag) or \eqref{AgD}.
\end{ex}

\begin{ex}[Cubic terms satisfying the condition {\rm (Ag)}]
We begin with simple examples satisfying (Ag).
When $N=1$, $F^{\rm c}(\pa u)=-(\pa_t u)^3$ is an example of the cubic terms 
satisfying (Ag) and \eqref{AgD}, as we have mentioned above.
Next we focus  on the case of two-component systems (i.e., $N=2$). 
Consider        
$$
F^{\rm c}(\pa u)=\left(
\begin{matrix}
F_1^{\rm c}(\pa u)\\
F_2^{\rm c}(\pa u)
\end{matrix}
\right)
=\left(
 \begin{matrix}
   -a (\pa_t u_1)^3+b (\pa_t u_1)(\pa_t u_2)^2\\
    -(\pa_t u_2)^3
  \end{matrix}
  \right)
$$
with $a\ge 0$ and $b\in \R$, whose reduced nonlinearity is
$$
F^{\rm c, red}(\omega, Y)=
\left(
\begin{matrix}
 aY_1^3-b Y_1Y_2^2\\
 Y_2^3
\end{matrix}
\right).
$$
\begin{itemize}
\item 
 If $a=0$ and $b\le 0$, we have
$$
Y\cdot F^{\rm c, red}(\omega, Y)
=-bY_1^2Y_2^2+Y_2^4\ge 0,\quad (\omega,Y) \in \Sph^1\times \R^2,
$$
whence the condition (Ag) is  satisfied with ${\mathcal A}(\omega)=I_2$.  
\item
If $a>0$, then the conditions (Ag) and \eqref{AgD} are satisfied for all 
$b\in \R$. Indeed, by choosing  
 $\dis{{\mathcal A}(\omega)=\left(
  \begin{matrix}
  1 & 0 \\
  0 & 1+(2a)^{-1}|b|^2
  \end{matrix}
  \right)}$, 
we have
\begin{align*}
Y\cdot {\mathcal A}(\omega)
F^{\rm c, red}(\omega, Y)
= &
\frac{a}{2}Y_1^4+Y_2^4 
+\left(\sqrt{\frac{a}{2}}Y_1^2-\sqrt{\frac{1}{2a}}|b|Y_2^2\right)^2\\
&+(|b|-b)Y_1^2Y_2^2.
\end{align*}
\end{itemize}
\end{ex}

\begin{ex}[A cubic term satisfying {\rm (Ag)} and \eqref{AgD} with a 
non-diagonal weight]
We give a bit less trivial example. Let $N=2$ and consider
$$
F^{\rm c}(\pa u)
    =\left(
    \begin{matrix}
    F_1^{\rm c}(\pa u)\\
    F_2^{\rm c}(\pa u)
    \end{matrix}
    \right)
$$
with 
\begin{align*}
 F_1^{\rm c}(\pa u)
 =&
   -(\pa_t u_1)^3-(\pa_t u_2)^3
   -\frac{1}{2}
    \Bigl( (\pa_1 u_1)^2 -(\pa_1 u_2)^2 \Bigr)
    \bigl( \pa_2 u_1-\pa_2 u_2 \bigr),\\
 F_2^{\rm c}(\pa u)
 =&
 (\pa_t u_1)^3-3(\pa_t u_1)^2(\pa_t u_2)\\
 &\hspace{13mm}  +
    \frac{1}{2}
   \Bigl((\pa_1 u_1)(\pa_2u_1)-(\pa_1 u_2)(\pa_2u_2)\Bigr)
   \bigl( \pa_1 u_1-\pa_1 u_2 \bigr),
\end{align*}
whose reduced nonlinearity is
$$
F^{\rm c, red}(\omega, Y)=
\left(
\begin{matrix}
 Y_1^3+Y_2^3-\omega_1^2\omega_2(Y_1^2-Y_2^2)(Y_1-Y_2)/2\\
 -Y_1^3+3Y_1^2Y_2+\omega_1^2\omega_2(Y_1^2-Y_2^2)(Y_1-Y_2)/2
\end{matrix}
\right).
$$
By choosing 
$$
{\mathcal A}(\omega)=
            4 \left(
              \begin{matrix}
                 2-\omega_1^2\omega_2 & 1-\omega_1^2\omega_2\\
                 1-\omega_1^2\omega_2 & 2-\omega_1^2\omega_2
              \end{matrix}
            \right)
        =
         2
         \left(\begin{matrix} 1 &  1 \\
                              1 & -1 \end{matrix}
         \right)
         \left(\begin{matrix}
                 3-2\omega_1^2\omega_2 & 0 \\
                 0 & 1
               \end{matrix}
         \right)
         \left(\begin{matrix} 1 & 1 \\
                              1 & -1 \end{matrix}
         \right),
$$
we get
\begin{align*}
& Y\cdot{\mathcal A}(\omega)F^{\rm c,red}(\omega,Y)\\
& \quad =
         \left( \begin{matrix} Y_1+Y_2 & Y_1-Y_2 \end{matrix} \right)
         \left(\begin{matrix}
                 3-2\omega_1^2\omega_2 & 0 \\
                 0 & 1
               \end{matrix}
         \right)\\
&\hspace{23mm}
         \times \left( \begin{matrix}
                (Y_1+Y_2)^3-(Y_1-Y_2)^3\\
                (Y_1-Y_2)^3+(3-2\omega_1^2\omega_2)(Y_1+Y_2)(Y_1-Y_2)^2
                \end{matrix}
         \right)
        \\
& \quad =
         (3-2\omega_1^2\omega_2)(Y_1+Y_2)^4+(Y_1-Y_2)^4.   
\end{align*}
Observing that $3-2\omega_1^2\omega_2\ge 1$ for $\omega\in \Sph^1$,
we see that (Ag) and \eqref{AgD} are satisfied for this cubic nonlinearity.
\end{ex}


\section{Preliminaries }\label{Prlmn}

\subsection{Commuting vector fields }

In this subsection, we recall basic properties of the vector fields 
associated with the wave equation. 
In what follows, we denote several positive constants by $C$ which may 
vary from one line to another. 
For  $y \in \R^d$ with a positive integer $d$, 
the notation $\jb{y}=(1+|y|^2)^{1/2}$ will be often used. 

We introduce
\begin{align*}
S:= t\pa_t+\sum_{j=1}^2 x_j \pa_j, \ %
L_1:= t\pa_1+x_1\pa_t, \ %
L_2:=t\pa_2+x_2\pa_t,\  %
\Omega:= x_1\pa_2-x_2\pa_1,
\end{align*}
and we set
$$
 \Gamma
  =(\Gamma_0,\Gamma_1, \ldots, \Gamma_6)
 :=(S, L_1, L_2, \Omega, \pa_0,\pa_1,\pa_2).
$$
With a multi-index $\alpha=(\alpha_0, \alpha_1, \ldots, \alpha_6)\in \ZP^7$, 
we write 
$\Gamma^\alpha
 =
 \Gamma_0^{\alpha_0}\Gamma_1^{\alpha_1}\cdots \Gamma_6^{\alpha_6}$, 
where $\ZP$ denotes the set of non-negative integers. 
For a smooth function $\psi=\psi(t,x)$ and a non-negative integer $s$, 
we define 
$$
|\psi(t,x)|_s=\sum_{|\alpha|\le s} |\Gamma^\alpha \psi(t,x)|,
\quad 
\|\psi(t)\|_s=\sum_{|\alpha|\le s} \|\Gamma^\alpha \psi(t,\cdot)\|_{L^2(\R^2)}.
$$
It is easy to see that $[\dal, L_j]=[\dal, \Omega]=[\dal, \pa_a]=0$ 
for $j=1,2$ and $a=0,1,2$, where $[A,B]=AB-BA$ for 
operators $A$ and $B$. 
We also have $[\dal, S]=2\dal$. 
Therefore for any $\alpha=(\alpha_0,\alpha_1,\ldots,\alpha_6)\in \ZP^7$ and 
a smooth function $\psi$, we have
\begin{equation}
\label{Comm01}
\dal \Gamma^\alpha \psi
=
\widetilde\Gamma^\alpha \dal \psi,  
\end{equation}
where 
 $\widetilde\Gamma^\alpha
  =(\Gamma_0+2)^{\alpha_0}\Gamma_1^{\alpha_1}\cdots\Gamma_6^{\alpha_6} $. 
We can check that 
$$
 [\Gamma_a, \Gamma_b]=\Sum_{0\le c\le 6} \Gamma_c, 
 \quad 
 [\Gamma_a, \pa_b]=\Sum_{0\le c\le 2} \pa_c. 
$$
Hence for any $\alpha, \beta\in \ZP^7$, 
and any non-negative integer $s$, 
there exist positive constants $C_{\alpha,\beta}$ and $C_s$ such that
\begin{align}
& |\Gamma^\alpha\Gamma^\beta\psi(t,x)|
  \le 
  C_{\alpha,\beta} |\psi(t,x)|_{|\alpha|+|\beta|},
\nonumber\\
 \label{Comm03}
& C_s^{-1} |\pa \psi(t,x)|_s
  \le 
  \sum_{0\le a\le 2}\sum_{|\gamma|\le s} |\pa_a\Gamma^\gamma\psi(t,x)|
  \le 
  C_s|\pa \psi(t,x)|_s
\end{align}
for any smooth function $\psi$.

For $x\in \R^2$, we use the polar coordinates $r=|x|$ and $\omega=|x|^{-1}x$, 
so that $x=r\omega$ and $\pa_r=\sum_{j=1}^2 (x_j/|x|)\pa_j$.
We put $\pa_\pm:=\pa_t\pm \pa_r$ and 
$D_\pm:=2^{-1} 
(\pa_r\pm \pa_t)$. 
We also introduce  
$$
 \hat\omega:=(-1,\omega_1,\omega_2).
$$ 
Remark that  
\begin{align}
 D_+=\frac{1}{2(t+r)} \Bigl( S+ \omega_1 L_1+\omega_2 L_2 \Bigr),
\label{D_plus}
\end{align}
which implies a gain of $(t+r)^{-1}$ in $D_+$ with the aid of $\Gamma$'s. 
Writing $\pa_j$ in the polar coordinates, we get
\begin{align}
|(\pa_j-\omega_j \pa_r) \psi(t,x)|\le & \frac{1}{r}|\Omega \psi(t,x)|
=\frac{1}{t+r}|(\Omega+\omega_1L_2-\omega_2L_1)\psi(t,x)|
\nonumber\\
\le & C\frac{|\Gamma \psi(t,x)|}{t+r},\quad j=1,2
\label{D_plus_a}
\end{align}
for a smooth function $\psi$.
Since $\pa_t=-D_-+D_+$ and $\pa_r=D_-+D_+$, it follows from
\eqref{D_plus} and \eqref{D_plus_a} that
\begin{equation}
|\pa \psi(t,x)-\hat{\omega} D_-\psi(t,x)|\le C \jb{t+r}^{-1}|\Gamma \psi(t,x)|.
\label{D_plus_+}
\end{equation}

Now we summarize a couple of useful lemmas which will be needed in the 
subsequent sections. 

\begin{lem}\label{Lem_Aux01}
Let 
$\Lambda_T:=\{(t,x)\in [0, T)\times \R^2; |x| \ge t/2\ge 1\}$. 
There is a positive constant $C$ such that
$$ 
 \Bigl|
  |x|^{1/2}\pa \psi(t,x)
  -
  \hat\omega D_-\bigl( |x|^{1/2}\psi(t,x) \bigr)
 \Bigr|
 \le 
 C \jb{t + |x| }^{-1/2}|\psi(t,x)|_1
$$
for $(t,x)\in \Lambda_T$ and $\psi\in C^1([0, T)\times \R^2)$.
\end{lem}

\begin{lem}\label{Lem_Hor}
{\rm{(i)}}\ Let $\psi$ be a smooth solution to
$$
\dal \psi(t,x)=G(t,x), \quad (t,x)\in (0, T)\times \R^2
$$
with initial data $\psi=\pa_t \psi=0$ at $t=0$. 
Then there exists a universal positive constant $C$, which is independent of $T$, 
such that
\begin{equation}
 \jb{t+|x|}^{1/2}|\psi(t,x)|
 \le 
 C\sum_{|\alpha|\le 1}\int_0^t
 \frac{\|\Gamma^{\alpha} G(\tau, \cdot)\|_{L^1(\R^2)}}{\jb{\tau}^{1/2}}d\tau 
\label{HorEst01}
\end{equation}
for $(t,x)\in [0,T)\times \R^2$.
\smallskip\\
{\rm{(ii)}}\ 
Let $\psi^0$ be a smooth solution to $\dal \psi^0(t,x)=0$ for 
$(t,x)\in (0,T)\times \R^2$ satisfying 
$\psi^0(0,x)=(\pa_t \psi^0)(0,x)=0$ for $|x|\ge R$ with some $R>0$. 
Then there is a positive constant $C_R$,  
depending only on $R$, such that 
\begin{equation}
 \jb{t+|x|}^{1/2}|\psi^0(t,x)|\le C_R \|\psi^0(0)\|_2,
 \quad (t,x)\in [0,T)\times \R^2.
\label{HorEst02}
\end{equation}
\end{lem}

\begin{lem}\label{Lem_Lind}
For any non-negative integer $s$, there exists a positive constant $C_s$ 
such that
$$
 |\pa \psi(t,x)|_s\le C_s \jb{t - |x| }^{-1}|\psi(t,x)|_{s+1},
 \quad (t,x)\in [0, T)\times \R^2
$$
for any $\psi\in C^{s+1}([0, T)\times \R^2)$.
\end{lem}

Lemma~\ref{Lem_Aux01} is a consequence of \eqref{D_plus_+}.
See \cite{KMuS} for its proof. 
The estimate \eqref{HorEst01} in Lemma~\ref{Lem_Hor} 
is often called H\"ormander's $L^1$--$L^\infty$ 
estimate, which is proved in \cite{Hor}. The estimate \eqref{HorEst02} for $t>1$ is
an easy consequence of \eqref{HorEst01} and the energy identity via the cut-off argument, while
 \eqref{HorEst02} for $0\le t\le 1$ follows from the energy identity and the Sobolev embedding theorem (see \cite{Lin08} for example).
Lemma~\ref{Lem_Lind} is due to Lindblad~\cite{Lind}; only the case of 
three space dimensions is treated in \cite{Lind}, but 
the two-dimensional case can be similarly proved 
(see \cite{KMuS} for instance).

\subsection{The null condition and the generalized energy estimate}

First we recall the estimates for quadratic  
terms satisfying the null condition \eqref{null2}.
Using \eqref{D_plus_+}, we have 
\begin{equation}
\label{NullEst01}
|Q_0(\phi,\psi)|+\sum_{a,b=0}^2 |Q_{ab}(\phi,\psi)|\le
C\jb{t+r}^{-1}(|\pa \phi|\,|\Gamma\psi|+|\Gamma\phi|\,|\pa \psi|),
\end{equation}
where $Q_0$ and $Q_{ab}$ are the null forms defined by \eqref{nullform_0}
and \eqref{nullform_ab}.
Since $\Gamma^\alpha \bigl(Q_0(\phi,\psi)\bigr)$ or 
$\Gamma^\alpha \bigl(Q_{ab}(\phi,\psi)\bigr)$ for any $\alpha\in \Z_+^7$ 
can be written as a linear combination of 
$Q_0(\Gamma^\beta \phi, \Gamma^\gamma \psi)$ 
and $Q_{cd}(\Gamma^\beta \phi, \Gamma^\gamma \psi)$ 
with $|\beta|+|\gamma|\le |\alpha|$ and $0\le c, d\le 2$, 
\eqref{Null2Express} and \eqref{NullEst01} yield the following lemma 
(see~\cite{Kla86} for the details):

\begin{lem}\label{NullEst02}
Suppose that \eqref{null2} is satisfied.
For 
$s\in \Z_+$, we have
$$
|F^{\rm q}(\pa u)|_s
 \le C \jb{ t + |x| }^{-1}
 \left(
   |\pa u||\Gamma u|_{s}
   +
   |\pa u|_{[s/2]}|\Gamma u|_{s-1}
   +
   |\Gamma u|_{[s/2]}|\pa u|_s
 \right)
$$
with a positive constant $C$. 
Here, $|\cdot|_{-1}$ is regarded as $0$.
\end{lem}

We must make use of this enhanced decay to treat $F^{\rm q}$ in the 
energy estimate. 
However, if we use Lemma~\ref{NullEst02} 
in the standard energy inequality, 
we need some estimate for $|\Gamma u|_s$ which does not follow form the 
standard energy inequality. 
To overcome this difficulty, we use a generalized energy inequality  
due to Alinhac~\cite{Ali01a} and \cite{Ali04}. We introduce 
$$
 Z=(Z_1,Z_2)
 =
 \left(\frac{x_1}{|x|}\pa_t+\pa_1, \frac{x_2}{|x|}\pa_t+\pa_2\right).
$$

\begin{lem}\label{GhostEnergy}
Let $T\in (0,\infty]$. 
Suppose that 
$\psi=\psi(t,x)$ is a smooth function satisfying
$$
\dal\psi(t,x)=G(t,x),\quad (t,x)\in (0,T)\times \R^2.
$$
For any $\rho>1$, there is a positive constant $C$, depending only on $\rho$, 
such that
\begin{align*}
& \|\pa \psi(t)\|_{L^2(\R^2)}^2
  +
  \int_0^t 
  \left( 
     \int_{\R^2}  \frac{|Z\psi(\tau, y)|^2 }{\jb{\tau-|y|}^\rho} dy 
  \right)
  d\tau\\
& \quad \le 
 C   
  \|\pa\psi(0)\|_{L^2(\R^2)}^2
  +
 C
 \int_0^t
   \left( \int_{\R^2} \left|G(\tau,y)(\pa_t\psi)(\tau,y)\right|dy \right)
 d\tau
\end{align*}
for $t\in [0,T)$.
\end{lem}

For the convenience of the readers, we will give the proof of this lemma 
in the appendix.


Next we introduce an auxiliary notation related to the operator $Z$.
For a non-negative integer $s$ and a smooth function $\psi$, we put 
$$
 |\psi(t,x)|_{Z,s}
 =
 \sum_{k=1}^2 \sum_{|\alpha|\le s} |Z_k \Gamma^\alpha \psi(t,x)|. 
$$
Observing that 
\begin{align*}
&S=x_1Z_1+x_2Z_2+(t-|x|)\pa_t, \\
&L_k=|x| Z_k + (t-|x|)\pa_k,  \quad k=1,2, \\
&\Omega=x_1Z_2-x_2Z_1,
\end{align*}
we can easily obtain the following lemma:

\begin{lem}\label{NullEst03}
For $s\in \ZP$, we have
$$
 |\Gamma \psi(t,x)|_s
 \le 
 C |x|\, |\psi(t,x)|_{Z,s}+\jb{t-|x|} |\pa \psi(t,x)|_s
$$
with a positive constant $C$.
\end{lem}

As a consequence, we have the following: 
\begin{cor} \label{Cor_NullEst} 
 Suppose that \eqref{null2} is satisfied. 
 For $s \in \ZP$, we have
$$
 |F^{\rm q}(\pa u)|_s
  \le 
  C \left(
     |\pa u||u|_{Z,s} + |\pa u|_{[s/2]}|u|_{Z,s-1}
   \right)
  +
  C \jb{ t + |x| }^{-1}   |u|_{[s/2]+1}|\pa u|_s
$$
with a positive constant $C$.
Here, $|\cdot|_{Z,-1}$ is regarded as $0$.
\end{cor}

\begin{proof}
By Lemmas~\ref{NullEst03} and \ref{Lem_Lind}, we have 
\begin{align*} 
 |\pa u||\Gamma u|_s
 &\le 
 C \jb{t+|x|} |\pa u||u|_{Z,s} + C\jb{t-|x|} |\pa u||\pa u|_{s}\\
  &\le 
 C \jb{t+|x|} |\pa u||u|_{Z,s} + C |u|_{1}|\pa u|_{s}
\end{align*}
as well as 
\begin{align*} 
 |\pa u|_{[s/2]} |\Gamma u|_{s-1}
  &\le 
   C \jb{t+|x|} |\pa u|_{[s/2]}|u|_{Z,s-1} 
    +
   C\jb{t-|x|} |\pa u|_{[s/2]}|\pa u|_{s-1}
  \\
  &\le 
   C \jb{t+|x|} |\pa u|_{[s/2]}|u|_{Z,s-1} 
   + 
   C |u|_{[s/2]+1}|\pa u|_{s-1}.
\end{align*}
The desired inequality follows immediately from them and Lemma~\ref{NullEst02}.
\end{proof}

\subsection{The profile equation}

Let $0<T\le \infty$, and let  $u$ be the solution to \eqref{eq}--\eqref{data}
on $[0, T)\times \R^2$.
We suppose that 
\begin{align}
 \supp f\cup \supp g \subset B_R 
\label{supp_0}
\end{align}
for some $R>0$, where $B_M=\{x\in \R^2;|x|\le M\}$ for $M>0$. 
Then, from the property of finite propagation, we have
\begin{align}
 \supp u(t,\cdot)\subset B_{t+R},\quad 0\le t<T.
\label{supp_t}
\end{align}

Now we put $r=|x|$, $\omega=(\omega_1, \omega_2)=x/|x|$ 
so that 
\begin{align} 
 r^{1/2}\Box \phi 
 =
 \pa_+ \pa_-(r^{1/2}\phi)
 -
 \frac{1}{4r^{3/2}}\left( 4 \Omega^2 + 1 \right) \phi.
 \label{dal_polar}
\end{align}
We define $U=(U_j)_{1\le j\le N}$   
by
\begin{align} 
U(t, x):=D_-\bigl(r^{1/2} u(t, x) \bigr),
\quad (t, x)\in [0,T)\times (\R^2\setminus\{0\})
 \label{U}
\end{align}
for the solution $u$ of \eqref{eq}. In view of Lemma~\ref{Lem_Aux01}, 
the asymptotic profile as 
$t \to \infty$ of $\pa u$  should be given by $ \hat{\omega} U/r^{1/2}$, 
because we can expect $|u(t,x)|_1\to 0$ as $t\to\infty$. 
Also it follows from \eqref{dal_polar} that 
\begin{align}
 \pa_+U(t, x)= -\frac{1}{2t} F^{\rm c, red}\bigl(\omega, U(t,x)\bigr)+H(t, x),
 \label{ode_0}
\end{align}
where $F^{\rm c, red}(\omega, Y)$ is defined by 
\eqref{ReducedNonlinearityC}, and $H=H(t,x)$ is given by
\begin{align}
\label{DefRemainder}
H=& 
-\frac{1}{2}\left(r^{1/2} F(\pa u)-\frac{1}{t}F^{\rm c, red}(\omega, U)\right)
-\frac{1}{8r^{3/2}} \left( 4\Omega^2+1 \right) u.
\end{align}
As Lemma~\ref{Lem_Remainder} below indicates, $H$ can be regarded as a 
remainder when \eqref{null2} is satisfied. For these reasons, we call \eqref{ode_0} {\em the profile equation} 
associated with \eqref{eq}, which plays an important role 
in our analysis. 

We also need an analogous equation for $\Gamma^\alpha u$ with a multi-index 
$\alpha\in \Z_+^7$. For this purpose, we put
\begin{equation}
\label{DefUAlpha}
  U^{(\alpha)}(t,x)
  := 
  D_-\bigl(r^{1/2} \Gamma^\alpha u(t,x)\bigr).
\end{equation}
Since 
$\dal(\Gamma^\alpha u)=\widetilde{\Gamma}^\alpha\left(F(\pa u)\right)$,
we deduce from \eqref{dal_polar} that 
\begin{align}
\pa_+U^{(\alpha)}= -\frac{1}{2t} G_{\alpha}(\omega, U, U^{(\alpha)})+H_\alpha
 \label{ode_alpha}
\end{align}
for $|\alpha|\ge 1$,
where 
$G_\alpha=(G_{\alpha,j})_{1\le j\le N}$    
and $H_\alpha$ are given by
$$
 G_{\alpha,j}\left(\omega, U, U^{(\alpha)}\right)
 = \sum_{k=1}^N\frac{\pa F_j^{\rm c, red}}{\pa Y_k}(\omega, U) U^{(\alpha)}_k 
$$
and 
\begin{align}
H_\alpha(t,x)
= &
-\frac{1}{2}\left(
 r^{1/2}\widetilde{\Gamma}^\alpha \bigl( F(\pa u)\bigr)
 - 
 \frac{1}{t}G_{\alpha}\left(\omega, U, U^{(\alpha)}\right)
 \right) 
  \nonumber\\
 &{}-\frac{1}{8r^{3/2}}\left( 4 \Omega^2 + 1 \right) \Gamma^\alpha u,
\label{DefHAlpha}
\end{align}
respectively. 

We close this section with preliminary estimates for 
$H$ and $H_{\alpha}$, in terms of the solution $u$, near the light cone. 
We put 
\begin{align*}
 \Lambda_{T,R}:=\{(t,x)\in [0, T)\times \R^2;\, 1\le t/2\le |x|\le t+R\}.
\end{align*}
Note that we have
$$
\jb{t+|x|}^{-1}\le |x|^{-1}\le 2t^{-1}\le 3(1+t)^{-1}\le 3(\jb{R}+2)\jb{t+|x|}^{-1}
$$
for $(t,x)\in \Lambda_{T, R}$.
In other words, the weights $\jb{t+|x|}^{-1}$, 
$(1+t)^{-1}$, $|x|^{-1}$ and $t^{-1}$ are equivalent to each other in 
$\Lambda_{T,R}$. 
For $s \in \ZP$, we also introduce an auxiliary notation 
$|\cdot |_{\sharp,s}$ by
\begin{equation}
|\phi(t,x)|_{\sharp,s}:=|\pa \phi(t,x)|_s+\jb{t+|x|}^{-1}|\phi(t,x)|_{s+1}.
\label{norm_sharp}
\end{equation}

\begin{lem}\label{Lem_Remainder}
Suppose that the null condition \eqref{null2} is satisfied.
There is a positive constant $C$, which is independent of $T$, such that 
\begin{equation}
|H(t,x)|
 \le  C 
 t^{-1/2}(1+t |\pa u|^2+|u|_{\sharp,0})|u|_{\sharp,0}|u|_1
 +
 C t^{-3/2}|u|_2
\label{Est_H}
\end{equation}
for $(t,x)\in \Lambda_{T,R}$, provided that $\sup_{(t,x)\in \Lambda_{T,R}}|\pa u(t,x)|$ is small enough. 
Also, for $s \ge 1$, there is a positive constant $C_s$,not depending on $T$, 
such that  
\begin{align}
 \sum_{|\alpha|=s}|H_\alpha(t,x)|
  \le 
  &C_s 
    t^{1/2}|\pa u|_{s-1}^3
    +
   C_s  t^{-1/2}(1+t|\pa u|_s^2+|u|_{\sharp,s})|u|_{\sharp,s}|u|_{s+1}
  \nonumber\\
  & + 
   C_s  t^{-3/2}|u|_{s+2}
  \label{Est_H_alpha}
\end{align}
for $(t,x)\in \Lambda_{T,R}$, provided that $\sup_{(t,x)\in \Lambda_{T,R}}|\pa u(t,x)|_{[s/2]}$
is small enough.
\end{lem}

\begin{proof}
Let $(t,x)=(t,r\omega)\in \Lambda_{T,R}$ in what follows.
We put
$$
F^{\rm h}(\pa u)=F(\pa u)-\bigl(F^{\rm q}(\pa u)+F^{\rm c}(\pa u)\bigr),
$$
so that we have $F^{\rm h}(\pa u)=O(|\pa u|^4)$ for small $\pa u$.

First we consider the estimate for $H$. We decompose it as follows: 
\begin{align*}
 H
 = &
 -\frac{1}{2}\Bigl( r^{1/2} F(\pa u) - r^{-1} F^{\rm c, red}(\omega, U) \Bigr) 
 {}-\frac{t-r}{2rt}F^{\rm c, red}(\omega, U)\\
 &{}-\frac{1}{8r^{3/2}} \left( 4 \Omega^2 + 1 \right) u.
\end{align*}
It is easy to see that the third term can be dominated by $Ct^{-3/2}|u|_2$. 
To estimate the second term, we note that 
\begin{align*}
|U|
 \le 
 r^{1/2}|D_-u|+\frac{C}{\jb{t+r}^{1/2}}|u|
 \le 
 C r^{1/2} |u|_{\sharp,0}
\end{align*}
and that 
\begin{align*}
 \jb{t-r}|U|
 \le 
 C t^{1/2} \left( 
  \jb{t-r}|\pa u|_{0}  +  \frac{\jb{t-r}}{\jb{t+r}}|u|_{0} 
 \right) 
 \le 
 C t^{1/2} |u|_{1},
\end{align*}
where we have used Lemma~\ref{Lem_Lind} to get the last inequality.
Then we obtain 
\begin{align*}
 \frac{|t-r|}{rt}|F^{\rm c, red}(\omega, U)|
 \le 
 C t^{-1} \jb{t-r} |U| \cdot \bigl(r^{-1/2}|U| \bigr)^{2}
 \le 
 C t^{-1/2} |u|_{1} |u|_{\sharp,0}^2.
\end{align*}
As for the  the first term, 
we deduce from Lemmas~\ref{NullEst02} and \ref{Lem_Aux01} that
\begin{align*}
 &|r^{1/2} F(\pa u) - r^{-1} F^{\rm c, red}(\omega, U) |\\
 & \le 
 |r^{1/2}F^{\rm q}(\pa u)|
 +
 |r^{1/2}F^{\rm h}(\pa u)|
 +
 |r^{1/2} F^{\rm c}(\pa u)-r^{-1}F^{\rm c,red}(\omega, U)|
 \\
 &\le 
  Ct^{-1/2} |u|_1|\pa u|+Ct^{1/2}|\pa u|^4\\
  &\phantom{\le\ } 
   + 
    \frac{C}{r}    
    \sum_{k,l,m}\sum_{a, b, c}  
  \left|
    (r^{1/2}\pa_a u_k)(r^{1/2}\pa_b u_l)(r^{1/2}\pa_c u_m)
    -
    (\omega_a U_k)(\omega_b U_l)(\omega_c U_m)
 \right|\\
&\le 
  Ct^{-1/2} |u|_1|u|_{\sharp,0}
  +
  Ct^{1/2}|\pa u|^2 |u|_{\sharp,0}|u|_1      
  \\
 &\phantom{\le\ }+
  C\bigl( |\pa u|+ r^{-1/2}|U| \bigr)^2|r^{1/2} \pa u-\hat{\omega} U|\\
&\le 
 Ct^{-1/2} (1+t |\pa u|^2+|u|_{\sharp,0})|u|_{\sharp,0}|u|_1.
\end{align*}
Summing up,  we arrive at \eqref{Est_H}.

Next we turn to the estimate for $H_{\alpha}$ with $|\alpha|=s \ge 1$. 
We set 
$\widetilde{F}^{\rm c}_\alpha
 =(\widetilde{F}^{\rm c}_{\alpha, j})_{1\le j\le N}$    
with 
\begin{align*}
 \widetilde{F}^{\rm c}_{\alpha, j}
 =\sum_{k,l,m=1}^N \sum_{a,b,c=0}^{2} 
 C_{jklm}^{abc} \bigl\{
  (\Gamma^\alpha \pa_a u_k)(\pa_b u_l)(\pa_c u_m)
  &+ 
  (\pa_a u_k)(\Gamma^\alpha \pa_b u_l)(\pa_c u_m)
  \\
  &+
  (\pa_a u_k)(\pa_b u_l)(\Gamma^\alpha \pa_c u_m)
 \bigr\}
\end{align*}
to split $H_{\alpha}$ into the following form: 
\begin{align*}
 H_{\alpha}
 =& 
  -\frac{r^{1/2}}{2}\widetilde{\Gamma}^{\alpha}
   \bigl(F^{\rm q}(\pa u)+F^{\rm h}(\pa u)\bigr)
  -
  \frac{r^{1/2}}{2} \left(\widetilde{\Gamma}^{\alpha} (F^{\rm c}(\pa u))
  -
  \widetilde{F}^{\rm c}_{\alpha} \right)\\ 
 &-\frac{1}{2}
 \left(r^{1/2} \widetilde{F}^{\rm c}_{\alpha} - r^{-1} G_{\alpha} \right) 
 -\frac{t-r}{2rt}G_{\alpha} 
 -\frac{1}{8r^{3/2}} \left( 4 \Omega^2 + 1 \right) \Gamma^{\alpha}u.
\end{align*}
The second term can be estimated by $Ct^{1/2} |\pa u|_{s-1}^3$, 
since it consists of a linear combination of the terms 
having the form 
$$
 r^{1/2}(\Gamma^{\beta}\pa_a u_k)(\Gamma^{\gamma}\pa_b u_l)
 (\Gamma^{\delta}\pa_c u_m)
$$ 
with $k,l,m\in \{1,\ldots, N\}$, $a, b, c \in \{0,1,2\}$, and  
$|\beta|, |\gamma|, |\delta| \le s-1$. 
Other four terms can be treated in the same way as in the previous case; 
they are dominated by 
$$
 C t^{-1/2}(1+t|\pa u|_s^2+|u|_{\sharp,s})|u|_{\sharp,s}|u|_{s+1} + Ct^{-3/2}|u|_{s+2}.
$$ 
Therefore we obtain \eqref{Est_H_alpha} as desired.
\end{proof}

\section{Proof of the small data global existence}\label{PSDGE}

The argument of this section is almost parallel to that of Section 5 in 
\cite{KMaS}, where quadratic semilinear systems of wave equations in $\R^3$ 
are considered. 
However, the argument becomes slightly more complicated 
because we are considering lower dimensional case here.


Let $u(t,x)$ be a smooth solution to 
\eqref{eq}--\eqref{data} on $[0, T_0)\times\R^2$ with some 
$T_0\in (0,\infty]$. For $0<T\le T_0$, we put 
\begin{align*}
 e[u](T)= \sup_{(t,x)\in[0,T)\times \R^2}
 \Bigl(
 &\jb{t+|x|}^{(1/2)-\mu}|u(t,x)|_{k+1}
  \\
 &+\jb{t+|x|}^{1/2} \jb{t-|x|}^{1-\mu}  |\pa u(t,x)| 
  \\
 &+ 
 \jb{t+|x|}^{(1/2)-\nu}  \jb{t-|x|}^{1-\mu}
 |\pa u(t,x)|_{k} 
 \Bigr)
\end{align*}
with some $\mu$, $\nu>0$ and a positive integer $k$.
We also put
$$
e[u](0)=\lim_{T\to +0} e[u](T).
$$
Observe that there is a positive constant $\eps_1$ such that 
$0<\eps\le \eps_1$ implies $e[u](0)\le \sqrt{\eps}/2$, 
because we have $e[u](0)=O(\eps)$.

The main step in the proof of Theorem~\ref{Thm_SDGE} is to show the following:

\begin{prop} \label{Prop_Apriori}
Let $k\ge 4$, $0<\mu < 1/10$ and $0< (8k+7) \nu \le \mu$. 
There exist positive constants $\eps_2$ and $M$, which depend only on 
$k$, $\mu$ and $\nu$,  such that 
\begin{align} 
 e[u](T)\le \sqrt{\eps}
\label{est_before}
\end{align} 
implies 
\begin{align} 
 e[u](T) \le M\eps, 
\label{est_after}
\end{align} 
provided that $0<\eps \le \eps_2$ and $0<T\le T_0$.
\end{prop}

Once the above proposition is obtained, 
the small data global existence for \eqref{eq}--\eqref{data} 
can be derived by the standard continuity argument: 
Let $T^*$ be the supremum of such $T\in (0,\infty)$ that the
Cauchy problem (\ref{eq})--(\ref{data}) admits a unique classical solution
$u\in C^\infty([0,T)\times \R^2;\R^N)$, and assume that $T^*<\infty$. 
Then, it follows from the standard blow-up criterion (see, e.g., \cite{sogge}) 
that 
\begin{align}
 \lim_{t\to T^*-0} \left(\|u(t,\cdot)\|_{L^\infty(\R^2)}+\|\pa u(t, \cdot)\|_{L^\infty(\R^2)}\right)=\infty.
 \label{blowup}
\end{align}
On the other hand, by setting
$$
T_*=\sup\left\{T\in [0,T^*)\,; e[u](T)\le \sqrt{\eps}  \right\},
$$
we can see that Proposition~\ref{Prop_Apriori} yields $T_*=T^*$,
provided that $\eps$ is small enough. 
Indeed, if $T_*<T^*$, then we have $e[u](T_*)\le\sqrt{\eps}$, and
Proposition~\ref{Prop_Apriori} implies that
$$
 e[u](T_*)\le M \eps \le \frac{\sqrt{\eps}}{2}
$$
for $0<\eps\le \min\left\{\eps_1, \eps_2, 1/(4M^2)\right\}$
(note that we have $T_*>0$ for $\eps\le \eps_1$). 
Then, by the continuity of the mapping $[0,T^*) \ni T \mapsto e[u](T)$, 
we can take $\delta>0$ such that 
$e[u](T_*+\delta) \le \sqrt{\eps}$,
which contradicts the definition of $T_*$, and we conclude that $T_*=T^*$. 

In particular, we have $e[u](T^*) \le \sqrt{\eps}$.
This implies that \eqref{blowup} never occurs for small $\eps$. 
In other words, we must have $T^*=\infty$, 
that is, the solution $u$ exists globally for small data. 
This completes the proof of Theorem~\ref{Thm_SDGE}.

From this proof, we see that
\begin{equation}
\label{AE00}
e[u](\infty) \le \sqrt{\eps}
\end{equation}
holds for the global solution $u$ with small $\eps$, and
Proposition~\ref{Prop_Apriori} again yields
\begin{equation}
 \label{AE}
 e[u](\infty)\le M\eps.
\end{equation}

\begin{proof}[Proof of Proposition~\ref{Prop_Apriori}.] 
In what follows, we always suppose that $0\le t<T$,
and that $0<\eps\le 1$. Let $R$ be the constant satisfying \eqref{supp_0}. 
Recall that we also have \eqref{supp_t} for the solution $u$. 
The proof of Proposition~\ref{Prop_Apriori} will be divided 
into several steps.
\medskip

\noindent{\bf Step 1: 
Basic energy estimates.}

We set
$$
 E_l(t)
 =
 \frac{1}{2}\|\pa u(t)\|_l^2
 +
 \frac{1}{2} \int_0^t\left(\int_{\R^2}
  \frac{|u(\tau,y)|_{Z,l}^2}{\jb{\tau-|y|}^{1+\mu}}
 dy\right)d\tau.
$$
The goal of this step is  to establish the following  estimates: 
\begin{align}
 E_l(t)^{1/2}
 \le 
 C \eps \jb{t}^{C_* \sqrt{\eps} + 2l\nu}
 \label{est_energy}
\end{align}
for $l\in \{0,1, \ldots, 2k+1\}$, 
where the constant $C_*$ is to be fixed. 

Let $0\le l\le 2k+1$.
In the sequel, we will use the following conventions:
$$
  |\pa u|_{-1}=0, \quad |u|_{Z,-1}=0, \quad\|\pa u\|_{-1}=0, \quad 
  E_{-1}(t)=0.
$$
From \eqref{Comm01}, \eqref{Comm03} and Lemma~\ref{GhostEnergy}, we get
\begin{align}
 E_l(t)
 \le 
  C_{1,l}  
  \|\pa u(0)\|_l^2
  +
  C_{1,l}
  \int_0^t\left(\int_{\R^2}
    \left|F\bigl(\pa u(\tau,y)\bigr)\right|_l|\pa u(\tau,y)|_ldy
  \right) d\tau,   
 \label{Ene01}
\end{align}
where $C_{1,l}$ is a positive constant depending only on $l$.
   %
It follows from \eqref{est_before} that 
\begin{align}
|F(\pa u)-F^{\rm q}(\pa u)|_l
 \le & 
 C_{l}\left(|\pa u|^2 |\pa u|_l+|\pa u|_{[l/2]}^2 |\pa u|_{l-1}\right) 
 \nonumber\\
 \le & 
  C_l{\eps}
     \jb{t}^{-1}|\pa u|_l
     +
   C_l{\eps}
   \jb{t}^{2\nu-1}|\pa u|_{l-1}
 \label{CubicEst}
\end{align}
with a positive constant $C_{l}$ depending only on $l$.
   %
   %
   %
By Corollary~\ref{Cor_NullEst}, we also have 
\begin{align}
 |F^{\rm q}(\pa u)|_l
 \le & 
  C_l \sqrt{\eps}                         
  \jb{t}^{-1/2}                     
  \jb{t-r}^{\mu-1}                        
  \left( |u|_{Z,l} + \jb{t}^{\nu} |u|_{Z,l-1} \right)
  \nonumber\\
 &\ +
  C_l \sqrt{\eps} \jb{t}^{\mu-(3/2)}|\pa u|_l.
  \label{QuadraticEst}
\end{align}
Since $\mu-1\le -(1+\mu)/2$, we deduce from \eqref{CubicEst} and 
\eqref{QuadraticEst} that 
\begin{align*}
  |F (\pa u)|_l |\pa u|_l
  \le &
   |F^{\rm q} (\pa u)|_l |\pa u|_l+ |F(\pa u)-F^{\rm q}(\pa u)|_l |\pa u|_l \\
  \le &
  C_l \sqrt{\eps}
  \jb{t-r}^{-(1+\mu)/2}  
  \left( |u|_{Z,l} + \jb{t}^{\nu} |u|_{Z,l-1} \right)
  \cdot \jb{t}^{-1/2} |\pa u|_l\\
 &
+
  C_l \sqrt{\eps} \jb{t}^{-1}|\pa u|_l^2\\
  & +
  C_l \eps^{3/4} \jb{t}^{(4\nu-1)/2}|\pa u|_{l-1} 
  \cdot \eps^{1/4} \jb{t}^{-1/2}|\pa u|_l\\
 \le &
  C_l \sqrt{\eps} \frac{|u|_{Z,l}^2 }{\jb{t-r}^{1+\mu} }
  +
  C_l \sqrt{\eps} \jb{t}^{-1}|\pa u|_l^2\\
 &
 +
  C_l \sqrt{\eps} \jb{t}^{2\nu} \frac{|u|_{Z,l-1}^2 }{\jb{t-r}^{1+\mu} }
  +
  C_l \eps^{3/2} \jb{t}^{4\nu-1}|\pa u|_{l-1}^2. 
\end{align*}
   %
   %
   %
   %
By integrating with respect to 
$(t,x)$,
and choosing $\eps$ suitably small, 
we 
get  
\begin{align}
 &C_{1,l} \int_0^t\left(\int_{\R^2}
   \left|F\bigl(\pa u(\tau,y)\bigr)\right|_l|\pa u(\tau,y)|_l
  dy \right) d\tau 
 \nonumber\\
 & \quad \le
 \frac{1}{2} E_l(t) 
 +
 C_{2,l} \sqrt{\eps} \int_0^t (1+\tau)^{-1} \|\pa u(\tau)\|_{l}^2 d\tau
 +
 C_{2,l} \sqrt{\eps} (1+t)^{2\nu} E_{l-1}(t)
 \nonumber\\
&\qquad + 
 C_{2,l} \eps^{3/2} \int_0^t (1+\tau)^{4\nu-1} \|\pa u(\tau)\|_{l-1}^2 d\tau
\label{Ene02}
\end{align}
with a positive constant $C_{2,l}$ depending only on $l$.

Now we put $C_*= 2 \max_{0\le l\le 2k+1} C_{2, l}$.  
   %
   %
We are going to prove \eqref{est_energy} by induction on $l$. 
In the case of $l=0$, it follows from \eqref{Ene01} and \eqref{Ene02} that
\begin{align*}
 E_0(t) 
 &\le 
 C\eps^2 +C_* \sqrt{\eps} 
 \int_0^t (1+\tau)^{-1}\|\pa u(\tau)\|_{0}^2 d\tau \\
 &\le
 C\eps^2+2C_*\sqrt{\eps}\int_0^t (1+\tau)^{-1}E_0(\tau)d\tau,
\end{align*}
whence the Gronwall lemma implies 
$$
 E_0(t)\le C\eps^2 (1+t)^{2C_* \sqrt{\eps}}.
$$
This shows \eqref{est_energy} for $l=0$.
Next we assume that \eqref{est_energy} holds for some 
$l\in \{0,1, \ldots, 2k\}$. 
Then it follows from \eqref{Ene01} and \eqref{Ene02} that
\begin{align*}
 E_{l+1}(t)
 \le & 
 C\eps^2
 +
 C_* \sqrt{\eps}
  \int_{0}^{t} (1+\tau)^{-1}\|\pa u(\tau)\|_{l+1}^2 d\tau
 \\
 &+C_*\sqrt{\eps}
   (1+t)^{2\nu}E_l(t)
   +
   C_* \eps^{3/2} \int_0^t (1+\tau)^{4\nu-1} \|\pa u(\tau)\|_{l}^2d\tau
\\
 \le &
 C\eps^2 
 + 
 2C_* \sqrt{\eps} \int_{0}^{t} (1+\tau)^{-1} E_{l+1}(\tau) d\tau 
 +
 C \eps^{5/2} (1+t)^{2C_*\sqrt{\eps}+4l\nu+2\nu}\\
 &+
 C \eps^{7/2} 
  \int_{0}^{t} (1+\tau)^{2C_*\sqrt{\eps}+ 4(l+1)\nu-1}d\tau\\
 \le & 
 C\eps^2 
 + 
 2C_* \sqrt{\eps} \int_{0}^{t} (1+\tau)^{-1} E_{l+1}(\tau)d\tau
 + 
 C \eps^{5/2} (1+t)^{2C_*\sqrt{\eps}+ 4(l+1)\nu},
\end{align*}
which yields
\begin{align*}
  E_{l+1}(t)
 &\le 
 C\eps^2 (1+t)^{2C_* \sqrt{\eps}}
 + 
 C\eps^{5/2} (1+t)^{2C_* \sqrt{\eps} + 4(l+1)\nu}\\
 &\le 
 C\eps^2 (1+t)^{2C_* \sqrt{\eps} + 4(l+1)\nu}.
\end{align*}
This means that \eqref{est_energy} remains true when $l$ is replaced by $l+1$, 
and \eqref{est_energy} has been proved for all $l\in \{0,1,\ldots, 2k+1\}$. 
\medskip

\noindent{\bf Step 2: 
Rough pointwise estimates. 
}

From now on, we assume  that $\eps\le (\nu/C_*)^2$. 
Then, since we have $k\ge 4$, 
it follows from \eqref{est_energy} with $l=2k+1$ that 
$$
 E_{k+5}(t)^{1/2}\le E_{2k+1}(t)^{1/2} \le C \eps \jb{t}^{(4k+3)\nu}.
$$
Observing that $2(4k+3) \nu \le \mu-\nu$ and $[(k+5)/2]\le k$, we get
\begin{align*}
\left\| \bigl|F\bigl(\pa u(t)\bigr)-F^{\rm q}\bigl(\pa u(t)\bigr)\bigr|_{k+5} \right\|_{L^1} 
  &\le  C \bigl\| |\pa u(t)|_{[(k+5)/2]} \bigr\|_{L^{\infty}} 
         \|\pa u(t)\|_{k+5}^2 \\
  &\le  
  \left(C \eps^{1/2} \jb{t}^{\nu-(1/2)}\right)
  \left(C \eps^2 \jb{t}^{2(4k+3)\nu}\right) \\
  &\le C \eps^{5/2} \jb{t}^{\mu-(1/2)},
\end{align*}
which yields
\begin{align}
\int_0^t
 \frac{\left\|
       \bigl|F\bigl(\pa u(\tau)\bigr)-F^{\rm q}\bigl(\pa u(\tau)\bigr)\bigr|_{k+5} \right\|_{L^1}
      }
      {\jb{\tau}^{1/2}}
 d\tau
 \le &
  C\eps^{5/2} \int_0^t\jb{\tau}^{\mu-1}d\tau
\nonumber\\
\le & C\eps^{5/2}\jb{t}^\mu.
 \label{CubicEst02}
\end{align}
On the other hand, it follows from Corollary~\ref{Cor_NullEst} that 
\begin{align*}
|F^{\rm q}(\pa u)|_{k+5} \le & C\sqrt{\eps}
\left(\jb{t}^{\nu-(1/2)}\jb{t-r}^{(3\mu-1)/2}\right)\jb{t-r}^{-(\mu+1)/2}|u|_{Z,k+5}
\\
&+C \sqrt{\eps}\jb{t}^{\mu-(3/2)}|\pa u|_{k+5}.
\end{align*}
Recalling \eqref{supp_t}, 
we deduce from the Schwarz inequality that 
\begin{align}
& \int_0^t 
   \frac{\left\| \bigl|F^{\rm q}\bigl(\pa u(\tau)\bigr)\bigr|_{k+5} \right\|_{L^1}}
        {\jb{\tau}^{1/2}}
  d\tau 
  \nonumber\\
& \quad \le 
  C\sqrt{\eps}\left( \int_0^t 
     \jb{\tau}^{2\nu-2} 
     \bigl\|\jb{\tau-|\cdot|}^{(3\mu-1)/2}\bigr\|_{L^2(B_{\tau+R})}^2
   d\tau \right)^{1/2}
   E_{k+5}(t)^{1/2}
   \nonumber\\
& \qquad\, +
   C\sqrt{\eps}\int_0^t 
    \jb{\tau}^{\mu-2} \|1\|_{L^2(B_{\tau+R})}\|\pa u(\tau)\|_{k+5}
  d\tau
  \nonumber\\
& \quad \le 
   C\eps^{3/2}
   \left(
     \jb{t}^{(3\mu+2\nu)/2+(4k+3)\nu} + \jb{t}^{\mu+(4k+3)\nu}
   \right)
 \nonumber\\
& \quad \le 
  C\eps^{3/2}\jb{t}^{3\mu}.
\label{QuadraticEst02}
\end{align}
By \eqref{CubicEst02}, \eqref{QuadraticEst02} and Lemma~\ref{Lem_Hor}, 
we have
\begin{align*}
 \jb{t+|x|}^{1/2} |u(t,x)|_{k+4} 
 \le & 
 C_R\|u(0)\|_{k+6}
 +
 C\int_0^t 
 \frac{\bigl\| |F(\pa u(\tau))|_{k+5} \bigr\|_{L^1}}{\jb{\tau}^{1/2}}d\tau\\
 \le & 
 C \eps + C\eps^{3/2}  \jb{t}^{3\mu} \\
 \le &
 C\eps \jb{t+|x|}^{3\mu}, 
\end{align*} 
that is,
\begin{align}
 |u(t,x)|_{k+4} \le C \eps \jb{t+ |x|}^{3\mu-(1/2)}
 \label{rough_bound_u_0}
\end{align}
for $(t,x) \in [0,T)\times \R^2$. 
Since $[(k+3)/2]\le k$, 
we see from Lemmas~\ref{NullEst02} and \ref{Lem_Lind} that
\begin{align*}
|F^{\rm q}(\pa u)|_{k+3}
 \le & 
  C \jb{t+r}^{-1}\left(|\pa u|_{k}   
  |\Gamma u|_{k+3}
  +
  |\Gamma u|_{k}                     
  |\pa u|_{k+3}\right)
 \\
 \le & 
  C \jb{t+r}^{-1}\jb{t-r}^{-1}|u|_{k+1}    
    |u|_{k+4}\nonumber\\
 \le & 
  C\eps^{3/2}\jb{t+r}^{4\mu-2}\jb{t-r}^{-1}\\
 \le &
  C\eps^{3/2}\jb{t}^{5\mu-2}\jb{t-r}^{-1-\mu}.
\end{align*}
Hence we get
\begin{align}
 \int_0^t 
  \frac{\left\|\bigl|F^{\rm q}\bigl(\pa u(\tau)\bigr)\bigr|_{k+3}\right\|_{L^1}}
       {\jb{\tau}^{1/2}}
 d\tau
 \le & 
  C\eps^{3/2}
  \int_0^t 
   \jb{\tau}^{5\mu-(5/2)}\|\jb{\tau-|\cdot|}^{-1-\mu}\|_{L^1(B_{\tau+R})}
  d\tau
 \nonumber\\
 \le & 
  C\eps^{3/2}\int_0^t \jb{\tau}^{5\mu-(3/2)} d\tau
 \nonumber\\
 \le &
 C\eps^{3/2},
 \label{QuadraticEst03}
\end{align}
because $5\mu-(3/2)<-1$.
Now it follows from \eqref{CubicEst02}, \eqref{QuadraticEst03} and 
Lemma~\ref{Lem_Hor} that
\begin{align*}
 \jb{t+|x|}^{1/2} |u(t,x)|_{k+2} 
 \le & 
 C_R\|u(0)\|_{k+4}
 +
 C\int_0^t 
 \frac{\bigl\| |F(\pa u(\tau))|_{k+3}\bigr\|_{L^1}}{\jb{\tau}^{1/2}}d\tau\\
 \le & 
 C \eps + C\eps^{3/2}  \jb{t}^{\mu}
 \nonumber\\
 \le &
 C\eps \jb{t+|x|}^{\mu}. 
\end{align*}
In other words, we obtain
\begin{equation}
|u(t,x)|_{k+2} \le C \eps \jb{t+|x|}^{\mu-(1/2)}
\label{rough_bound_u}
\end{equation}
for $(t,x)\in [0,T)\times \R^2$.
By Lemma~\ref{Lem_Lind}, we also have 
\begin{align}
 |\pa u(t,x)|_{k+1} \le C \eps \jb{t+|x|}^{\mu-(1/2)} \jb{t-|x|}^{-1}
\label{rough_bound_pa_u}
\end{align}
for $(t,x) \in [0,T)\times \R^2$. 
\medskip

 \noindent{\bf Step 3: 
 Estimates for $|\pa u(t,x)|_{k}$ away from the light cone.}

Now we put 
$\Lambda_{T,R}^{\rm c}:=\bigl([0,T)\times \R^2\bigr)\setminus \Lambda_{T,R}$.
In the case of $t/2<1$ or $|x|< t/2$, we see that 
$$
\jb{t-|x|}\le \jb{t+|x|}\le C\jb{t-|x|}.
$$
On the other hand, it follows from \eqref{supp_t} that $u(t,x)=0$ if $|x|>t+R$.
Hence \eqref{rough_bound_pa_u} implies
\begin{align}
 & \sup_{(t,x)\in \Lambda_{T,R}^{\rm c}} 
 \jb{t+|x|}^{1/2} \jb{t-|x|}^{1-\mu}|\pa u(t,x)|_{k}\nonumber\\
 & \qquad \le 
  \sup_{(t,x)\in\Lambda_{T,R}^{\rm c}}
  \jb{t+|x|}^{(3/2)-\mu}|\pa u(t,x)|_{k}
\le C\eps.
\label{est_after_01}
\end{align}
\medskip

\noindent{\bf Step 4:  
Estimates for $|\pa u(t,x)|$ near the light cone.}

Let $(t,x) \in \Lambda_{T,R}$.
We may assume $T\ge 2$, because $\Lambda_{T,R}$ is empty for $T<2$.
Remember that $t^{-1}$, $r^{-1}$, $\jb{t}^{-1}$ and $\jb{t+r}^{-1}$ are 
equivalent to each other in $\Lambda_{T,R}$. 
We define $U$, $U^{(\alpha)}$, $H$, $H_\alpha$ and $|\,\cdot\,|_{\sharp, s}$ 
as in the previous section 
(see \eqref{U}, \eqref{DefRemainder}, \eqref{DefUAlpha},
\eqref{DefHAlpha} and \eqref{norm_sharp}). 
We see from  \eqref{rough_bound_u} and \eqref{rough_bound_pa_u} that 
\begin{align}
 |u(t,x)|_{\sharp,k}\le C\eps t^{\mu-1/2}\jb{t-|x|}^{-1}.
 \label{est_u_sharp01}
\end{align}
By \eqref{Comm03}, 
\eqref{rough_bound_u} 
and Lemma~\ref{Lem_Aux01},  we have
\begin{align}
t^{1/2}|\pa u(t,x)|_l
 \le & 
 C \sum_{|\alpha|\le l} \bigl|\,|x|^{1/2} \pa \Gamma^\alpha u(t,x) \bigr| 
 \nonumber\\
 \le & 
 C\sum_{|\alpha|\le l} |U^{(\alpha)}(t,x)| + C\eps t^{\mu-1}
 \label{est_pa_u_02}
\end{align}
for $l\le k$. 
Also, it follows from \eqref{est_before}, \eqref{rough_bound_u}, \eqref{est_u_sharp01} and 
Lemma~\ref{Lem_Remainder} that 
\begin{align}
|H(t,x)|
 &\le 
 C\eps^2 t^{2\mu-(3/2)}\jb{t-|x|}^{-1} + C \eps t^{\mu-2}
 \nonumber\\
 &\le 
 C\eps t^{2\mu-(3/2)}\jb{t-|x|}^{-\mu-(1/2)}.
 \label{est_H}
\end{align}
Next we put 
$$
\Sigma=\left\{(t,x) \in 
[0,\infty)\times \R^2;\, |x|\ge \frac{t}{2}=1\ \text{or}\ |x|=\frac{t}{2}\ge 1 \right\}
$$
\scalebox{0.9}[1]{and we define $t_{0, \sigma}=\max\{2, -2\sigma\}$.  
What is important here is that the half line 
$\left\{(t, (t+\sigma)\omega);\, 
t\ge 0
\right\}$} \\
meets $\Sigma$ at the point 
$\left(t_{0, \sigma}, (t_{0,\sigma}+\sigma)\omega\right)$ for each 
fixed $(\sigma,\omega) \in \R 
\times \Sph^1$, so that
$$
\Lambda_{T,R}=\bigcup_{(\sigma, \omega)\in (-T/2, R]\times \Sph^1}\left\{
\left(t, (t+\sigma)\omega\right);\, t_{0,\sigma} \le t<T
\right\}.
$$
We also remark that 
\begin{align}
 C^{-1}\jb{\sigma} \le t_{0,\sigma} \le C \jb{\sigma}, \quad \sigma\le R
 \label{t_zero}
\end{align}
with a positive constant $C$ depending only on $R$.
When $(t,x)\in \Sigma$, 
we have $t^{\mu}\le C\jb{t-|x|}^{\mu}$. 
So it follows from \eqref{Comm03}, \eqref{est_u_sharp01} and 
Lemma~\ref{Lem_Aux01} that 
\begin{align}
 \sum_{|\alpha|\le k}|U^{(\alpha)}(t,x)|
 \le & 
 C t^{1/2} |u(t,x)|_{\sharp, k} \nonumber\\
 \le &
 C\eps \jb{t-|x|}^{\mu-1},
\quad (t,x)\in \Sigma \cap \Lambda_{T,R}.
\label{est_U_alpha}
\end{align}

Let ${\mathcal A}$ be the matrix in the condition (Ag).
Since ${\mathcal A}$ is positive-definite and continuous on $\Sph^1$,
we can find a positive constant $M_0$ such that 
\begin{equation}
 \label{Bound_A}
 M_0^{-1} |Y|^2 
 \le 
 Y \cdot {\mathcal A}(\omega)Y 
 \le 
 M_0 |Y|^2, 
 \quad 
 (\omega, Y) \in \Sph^1\times \R^N.
\end{equation}
Now we define
\begin{align}
V(t; \sigma,\omega)
 =
 U\bigl(t, (t+\sigma)\omega\bigr)
\label{profile_v}
\end{align}
for $0\le t< T$ and 
$(\sigma,\omega) \in \R\times \Sph^1$. In what follows, we fix $(\sigma, \omega)\in 
(-
T/2, R]\times \Sph^1$ and write $V(t)$ for $V(t;\sigma, \omega)$. 
Then, since the profile equation \eqref{ode_0} is rewritten as 
\begin{align}
 \frac{\pa V}{\pa t}(t)=(\pa_+U)\bigl(t, (t+\sigma)\omega\bigr)
 = 
 -\frac{1}{2t}F^{\rm c, red}\bigl(\omega, V(t)\bigr)+H\bigl(t, (t+\sigma)\omega\bigr) 
 \label{ode_V}
\end{align}
for $t_{0,\sigma} < t< T$, it follows from the condition {\rm (Ag)} that 
\begin{align}
 \frac{\pa}{\pa t}\bigl( V(t) \cdot {\mathcal A}(\omega)V(t) \bigr)
 &=
 2 V(t) \cdot {\mathcal A}(\omega)\frac{\pa V}{\pa t}(t) 
 \nonumber\\
 &=
 2V(t) \cdot {\mathcal A}(\omega)
 \left(
   - \frac{1}{2t}F^{\rm c, red}\bigl(\omega, V(t)\bigr)
   + H\bigl(t, (t+\sigma)\omega\bigr)
 \right)\nonumber\\ 
 &\le
  2 V(t) \cdot {\mathcal A}(\omega) H\bigl(t,(t+\sigma)\omega\bigr)
    \nonumber\\
 &\le 
  C\sqrt{V(t) \cdot {\mathcal A}(\omega)V(t)}\left|H(t,\bigl(t+\sigma)\omega\bigr)\right|
\label{star02}
\end{align}
for $t_{0,\sigma} < t< T$.
We also note that \eqref{est_U_alpha} for $k=0$ can be interpreted as
\begin{align}
 |V(t_{0,\sigma})|
 =
 \left|U\bigl(t_{0,\sigma}, (t_{0,\sigma}+\sigma)\omega\bigr)\right|
 \le 
 C\eps \jb{\sigma}^{\mu-1}.
 \label{est_V_ini}
\end{align}
We deduce from \eqref{est_H}, \eqref{t_zero}, \eqref{Bound_A}, \eqref{star02} and
\eqref{est_V_ini} that
\begin{align}
 |V(t)| & \le \sqrt{M_0} \sqrt{V(t) \cdot {\mathcal A}(\omega)V(t)} 
 \nonumber\\
 &\le C\left(
 \sqrt{V(t_{0,\sigma}) \cdot {\mathcal A}(\omega)V(t_{0,\sigma})} 
 +
 \int_{t_{0,\sigma}}^{t} \left|H\bigl(\tau,(\tau + \sigma)\omega\bigr)
\right|d\tau\right)
 \nonumber\\
 &\le 
 C\eps \jb{\sigma}^{\mu-1}
 +
 C\eps\jb{\sigma}^{-\mu-(1/2)}
 \int_{t_{0,\sigma}}^{t} {\tau}^{2\mu-(3/2)} d\tau
 \nonumber\\
 &\le 
 C\eps \jb{\sigma}^{\mu-1}
 \bigl\{ 1 + (\jb{\sigma}/t_{0,\sigma})^{(1/2)-2\mu} \bigr\}
 \nonumber\\
 &\le 
 C\eps \jb{\sigma}^{\mu-1}
\label{est_V}
\end{align}
for $t\ge t_{0,\sigma}$, 
where $C$ is  independent of $\eps$, $\sigma$ and $\omega$. 
\eqref{est_V} implies 
$$
 |U(t,x)|=|V(t; |x|-t,x/|x|)|\le C\eps\jb{t-|x|}^{\mu-1},
\quad (t,x)\in \Lambda_{T,R}.
$$
Finally, in view of \eqref{est_pa_u_02} with $l=0$, we obtain
\begin{align}
  \sup_{(t,x)\in \Lambda_{T,R}} 
  \jb{t+|x|}^{1/2}\jb{t-|x|}^{1-\mu} |\pa u(t,x)|
  \le C\eps.
  \label{est_after_02}
\end{align}
\medskip

\noindent{\bf Step 5: 
Estimates for $|\pa u(t,x)|_{k}$ near the light cone.}

For a nonnegative integer $s$, we set 
$$
{\mathcal U}^{(s)}(t,x):=\sum_{|\beta|\le s}|U^{(\beta)}(t,x)|.
$$
Let $1\le |\alpha|\le k$ and $(t,x)\in\Lambda_{T,R}$. 
By \eqref{est_pa_u_02} we get
\begin{align}
 |\pa u(t,x)|_{|\alpha|-1}
 \le 
 C    
 t^{-1/2}{\mathcal U}^{(|\alpha|-1)}(t,x)
 +
 C \eps t^{\mu-3/2}.   
 \label{est_pa_u_03}
\end{align}
It follows from \eqref{rough_bound_u}, \eqref{rough_bound_pa_u}, \eqref{est_u_sharp01}, 
\eqref{est_pa_u_03} and Lemma~\ref{Lem_Remainder} that
\begin{align}
|H_\alpha(t,x)|
 \le& 
 C   
  (1+\eps^2t^{2\mu}+\eps t^{\mu-(1/2)})
   \eps^2 t^{2\mu-(3/2)}\jb{t-|x|}^{-1}
  \nonumber\\
  & + 
  C \eps t^{\mu-2}
  +
  C \eps^3t^{3\mu-4}
  +
  C t^{-1} \bigl({\mathcal U}^{(|\alpha|-1)}(t,x)\bigr)^3
 \nonumber\\
 \le& 
 C\eps t^{
  4\mu-(3/2)}\jb{t-|x|}^{-
  3\mu-(1/2)}
 +
 Ct^{-1}\bigl({\mathcal U}^{(|\alpha|-1)}(t,x)\bigr)^3.
 \label{est_H_alpha}
\end{align} 
We put
$$
 V^{(\alpha)}(t; \sigma, \omega)=U^{(\alpha)}\bigl(t, (t+\sigma)\omega\bigr)
$$
for $0 \le t< T$ and $(\sigma,\omega) \in (-\infty, R]\times \Sph^1$.
We fix $(\sigma, \omega)\in (-
T/2, R]\times \Sph^1$ and write
$V^{(\alpha)}(t)$ for $V^{(\alpha)}(t; \sigma, \omega)$.
Then \eqref{ode_alpha} is rewritten as 
\begin{align*}
\frac{\pa V^{(\alpha)}}{\pa t}(t)
= -\frac{1}{2t} G_\alpha\left(\omega, V(t), V^{(\alpha)}(t)\right)
+H_\alpha\bigl(t, (t+\sigma) \omega\bigr)
\end{align*}
for $t_{0,\sigma} < t< T$. Hence by \eqref{est_V} and \eqref{est_H_alpha} 
we obtain
\begin{align*}
 \frac{\pa}{\pa t}\bigl|V^{(\alpha)}(t)\bigr|^2
 \le& 
 \frac{C}{t}  |V(t)|^2 \bigl|V^{(\alpha)}(t)\bigr|^2
 +2\left|H_\alpha\bigl(t, (t+\sigma) \omega\bigr)\right|\, \bigl|V^{(\alpha)}(t)\bigr|\\
 \le& 
 \frac{2C^* \eps^2}{t} \bigl|V^{(\alpha)}(t)\bigr|^2\\
 &+
 C\left(
  \eps t^{
  4\mu-(3/2)}\jb{\sigma}^{
  -3\mu-(1/2)}
  + 
  t^{-1}\bigl({\mathcal V}^{(|\alpha|-1)}(t)\bigr)^3
 \right)\, \bigl|V^{(\alpha)}(t)\bigr|,
\end{align*}
where 
$$
 {\mathcal V}^{(s)}(t)\bigl(={\mathcal V}^{(s)}(t;\sigma,\omega)\bigr)
 :=
 \sum_{|\beta|\le s} \bigl|V^{(\beta)}(t;\sigma, \omega)\bigr|,
$$
and $C^*$ is a positive constant independent of $\alpha$. 
Therefore it follows from \eqref{t_zero} and \eqref{est_U_alpha} that
\begin{align*}
 t^{-C^*\eps^2}|V^{(\alpha)}(t)|
 \le& 
 t_{0,\sigma}^{-C^*\eps^2}
 \bigl|V^{(\alpha)}(t_{0,\sigma})\bigr|\\
& +
 C\eps \jb{\sigma}^{
 -3\mu-(1/2)} \int_{t_{0,\sigma}}^t \tau^{-C^*\eps^2+
 4\mu-(3/2)}d\tau
 \\
 &{}+
 C\int_{t_{0,\sigma}}^t \tau^{-C^*\eps^2 -1}
 \bigl({\mathcal V}^{(|\alpha|-1)}(\tau)\bigr)^3 d\tau
 \\
 \le & 
 C\eps\jb{\sigma}^{\mu-1}
 +
 C\int_{2}^t \tau^{-C^*\eps^2-1}
 \bigl({\mathcal V}^{(|\alpha|-1)}(\tau)\bigr)^3 d\tau.
\end{align*}
By this inequality for $1\le |\alpha|\le l$ and \eqref{est_V},
we have
$$
 t^{-C^*\eps^2}{\mathcal V}^{(l)}(t)
 \le 
  C\eps\jb{\sigma}^{\mu-1}+C\int_{2}^t \tau^{-C^*\eps^2-1}
  \bigl({\mathcal V}^{(l-1)}(\tau)\bigr)^3 d\tau
$$
for $l \in \{1, \ldots, k\}$. 
Using this inequality, we can show inductively that
\begin{align}
 {\mathcal V}^{(l)}(t)
 \le 
 C \eps \jb{\sigma}^{\mu-1} t^{3^{l-1} C^* \eps^2} 
\label{est_V_l}
\end{align}
for $t_{0,\sigma}\le t < T$ and $l \in \{1, \ldots, k\}$. 
Indeed, we already know that 
$$
 {\mathcal V}^{(0)}(t)= |V(t)|
\le C\eps\jb{\sigma}^{\mu-1} 
$$
by \eqref{est_V}. Hence we have 
\begin{align*}
 t^{-C^*\eps^2}{\mathcal V}^{(1)}(t)
 \le  
 C\eps\jb{\sigma}^{\mu-1}
 +
 C\eps^3\jb{\sigma}^{3\mu-3}
 \int_{2}^\infty \tau^{-C^*\eps^2-1} d\tau
 \le 
 C\eps\jb{\sigma}^{\mu-1},
\end{align*}
which implies \eqref{est_V_l} for $l=1$. 
Next we suppose that \eqref{est_V_l} is true for some 
$l \in \{1, \ldots, k-1\}$. Then we have
\begin{align*}
 t^{-C^*\eps^2}{\mathcal V}^{(l+1)}(t)
 \le& 
  C\eps\jb{\sigma}^{\mu-1}+C \eps^3\jb{\sigma}^{3\mu-3}
  \int_{2}^t \tau^{(3^{l}-1)C^*\eps^2-1} d\tau\\
 \le& 
  C\eps \jb{\sigma}^{\mu-1} t^{(3^{l}-1)C^* \eps^2},
\end{align*}
which yields \eqref{est_V_l} with $l$ replaced by $l+1$. 
Hence \eqref{est_V_l} for $l\in \{1,\ldots,k\}$ 
has been proved.

By \eqref{est_pa_u_02} and \eqref{est_V_l} with $l=k$, we have 
\begin{align*}
|\pa u(t,x)|_k
\le 
C\eps \jb{t+|x|}^{-1/2+3^{k-1}C^*\eps^2}\jb{t-|x|}^{-1+\mu},
\quad (t,x)\in \Lambda_{T,R}.
\end{align*}
Finally we take $\eps$ so small that  $3^{k-1}  C^* \eps^2 \le \nu$. 
Then we obtain
\begin{align}
 \sup_{(t,x)\in \Lambda_{T, R}} 
 \jb{t+|x|}^{1/2-\nu}\jb{t-|x|}^{1-\mu}|\pa u(t,x)|_k
  \le C\eps.
 \label{est_after_03}
\end{align}
\medskip
 
\noindent{\bf The final step.}

By \eqref{rough_bound_u}, \eqref{est_after_01}, \eqref{est_after_02} and \eqref{est_after_03}, 
we can find two positive constants 
$\eps_2$ and $M$ such that \eqref{est_after} holds for $0<\eps\le \eps_2$. 
This completes the proof of Proposition~\ref{Prop_Apriori}.
\end{proof}

\section{Proof of the energy decay}\label{PED}

Before we proceed to the proof of Theorem~\ref{Thm_EnergyDecay}, 
we introduce a useful lemma. 

\begin{lem}  \label{Lem_Mats}
Let $C_0>0$, $C_1\ge 0$, $p>1$, $q>1$ and $t_0\ge 2$. Suppose that 
$\Phi(t)$ satisfies 
$$
 \frac{d \Phi}{dt}(t) \le \frac{-C_0}{t} |\Phi(t)|^p +\frac{C_1}{t^q}
$$ 
for $t\ge t_0$. Then we have 
$$
 \Phi(t) \le \frac{C_2}{(\log t)^{p^*-1}}
$$
for $t\ge t_0$, where $p^*$ is the H\"older conjugate of $p$ 
{\rm(}i.e., $1/p+1/p^*=1${\rm)}, and 
$$
 C_2
 =\frac{1}{\log 2}\left(
 (\log t_0)^{p^*}\Phi(t_0) 
 + 
 C_1\int_2^{\infty} \frac{(\log \tau)^{p^*}}{\tau^{q}}d\tau
 \right)+ 
 \left(\frac{p^*}{C_0 p} \right)^{p^*-1}.
$$
\end{lem}

\begin{rem}
Special cases of this lemma  have been used in 
Section 4 of \cite{KLS} and Section 5 of \cite{KimS} less explicitly.
\end{rem}

\begin{proof}
It follows from the Young inequality that 
\begin{align*}
 |\Phi(t)| 
 &=
  \Bigl( \kappa  (\log t) |\Phi(t)|^p \Bigr)^{\frac{1}{p}}
  \cdot
  \left( \frac{1}{(\kappa \log t)^{p^*-1}} \right)^{\frac{1}{p^*}}\\
 &\le
 \frac{\kappa}{p}(\log t) |\Phi(t)|^p
 +
 \frac{1}{p^* (\kappa \log t)^{{p^*-1}}}
\end{align*}
for $\kappa>0$. By choosing $\kappa=C_0 p/p^*$, we have 
$$
 p^* \frac{(\log t)^{p^*-1}}{t} \Phi(t)
 \le 
 (\log t)^{p^*} \frac{C_0}{t} |\Phi(t)|^p
 +
 \left(\frac{p^*}{C_0 p} \right)^{p^*-1} \frac{1}{t},
$$
whence
\begin{align*}
 \frac{d}{dt} \Bigl( (\log t)^{p^*} \Phi(t) \Bigr)
 &=
   (\log t)^{p^*} \frac{d \Phi}{dt}(t) 
   + 
   p^* \frac{(\log t)^{p^* -1}}{t} \Phi(t)\\
 &\le
   (\log t)^{p^*} 
   \left\{
     \frac{d \Phi}{dt}(t) +  \frac{C_0}{t} |\Phi(t)|^p
   \right\}
   +   
   \left(\frac{p^*}{C_0 p} \right)^{p^*-1} \frac{1}{t}\\
 &\le 
  C_1 \frac{(\log t)^{p^*}}{t^q} 
  +   
   \left(\frac{p^*}{C_0 p} \right)^{p^*-1} \frac{1}{t}.
\end{align*}
Integration with respect to $t$ implies 
\begin{align*}
 (\log t)^{p^*} \Phi(t) 
 &\le 
  (\log t_0)^{p^*} \Phi(t_0)
  + 
  C_1 \int_{t_0}^{t} \frac{(\log \tau)^{p^*}}{\tau^q} d\tau
  +   
  \left(\frac{p^*}{C_0 p} \right)^{p^*-1} \log \left( \frac{t}{t_0} \right) \\
 &\le 
  C_2 \log t,
\end{align*}
from which we deduce the desired inequality. 
\end{proof}

Now we are ready to finish the proof of Theorem~\ref{Thm_EnergyDecay}. 
Note that all the estimates in the proof of Proposition~\ref{Prop_Apriori} are
valid with $T=\infty$, because of \eqref{AE00}.
Let the assumptions of Theorem~\ref{Thm_EnergyDecay} be fulfilled.
The conditions (Ag) and \eqref{AgD} imply
$$
 \min_{|Y|=1,\, \omega\in \Sph^1} 
 Y\cdot {\mathcal A}(\omega)F^{\rm c, red}(\omega, Y)>0.
$$
Hence, in view of \eqref{Bound_A}, we can choose $C_0>0$ such that
$$
 Y \cdot \mathcal{A}(\omega) F^{\rm{c, red}}(\omega,Y) \ge 
 C_0 (Y \cdot \mathcal{A}(\omega) Y)^2, 
 \quad
 (\omega,Y)  \in \Sph^1 \times \R^N. 
$$
For $(t,x) \in \Lambda_{T,R}$, we fix $\sigma=|x|-t$, $\omega=x/|x|$ 
and set $\Phi(t)=V(t) \cdot \mathcal{A}(\omega) V(t)$ 
with $V(t)=V(t;\sigma,\omega)$ defined by \eqref{profile_v}. 
By \eqref{est_H}, \eqref{ode_V} and \eqref{est_V}, we get
$$
 \frac{d \Phi}{dt}(t)=2V(t)\cdot{\mathcal A}(\omega)\frac{dV}{dt}(t)
 \le 
 \frac{-C_0}{t}(\Phi(t))^2 
 + 
 \frac{C' \eps^2 \jb{\sigma}^{-3/2}}{t^{(3/2)-2\mu}}
$$
for $t\ge t_{0,\sigma}$ (cf.~\eqref{star02}), where $C'$ 
is a positive constant independent of $t$, $\sigma$, $\omega$ and $\eps$.
Therefore we can apply Lemma~\ref{Lem_Mats} to obtain 
\begin{equation}
 \Phi(t) \le \frac{C_{\sigma, \omega}}{\log t}, \quad t\ge t_{0,\sigma}
\label{Est_Phi}
\end{equation}
with
$$
C_{\sigma,\omega}=
 \frac{1}{\log 2}
  \left(
    (\log t_{0,\sigma})^2 \Phi(t_{0,\sigma})
    +
     C' \eps^2 \jb{\sigma}^{-3/2}  
    \int_{2}^\infty \frac{(\log \tau)^2}{\tau^{(3/2)-2\mu}}d\tau
   \right)
  +
  \frac{1}{C_0}.
$$
By \eqref{t_zero} and \eqref{est_V_ini}, we can find a positive constant 
$C_{3}$, not depending on $\eps$, $\sigma$, $\omega$ and $T$,  
such that $C_{\sigma,\omega}\le C_3$ for all 
$(\sigma,\omega)\in (-
T/2,R]\times \Sph^1$. 
Hence \eqref{Bound_A} and \eqref{Est_Phi} lead to 
$$
  |V(t;\sigma,\omega)|
 \le 
 \sqrt{M_0 \Phi(t)}
 \le 
 \frac{C}{\sqrt{\log t}}, \quad t\ge t_{0,\sigma},
$$
which, together with \eqref{est_after_01} and \eqref{est_pa_u_02}, yields
\begin{equation}
|\pa u(t,x)|\le C t^{-1/2}(\log t)^{-1/2},\quad (t,x)\in [2,\infty)\times \R^2.
\label{est_u_log}
\end{equation}
By \eqref{AE} and \eqref{est_u_log}, we have
\begin{equation}
\label{DecayDamping02}
 |\pa u(t,x)|
 \le 
 C t^{-1/2} \min \bigl\{ (\log t)^{-1/2}, \eps(R+|t-|x||)^{\mu-1} \bigr\}
\end{equation}
for $(t, x)\in [2,\infty)\times \R^2$.

  %
Given $\delta>0$, we put $\delta_0=\min\{\delta,1/37\}$ and
$\rho(t;\eps)=(\eps^2 \log t)^{(1/2)+2\delta_0}$.
We choose $\mu=4\delta_0/(1+4\delta_0)$ in the definition of 
$e[u](T)$. 
Let $t\ge 2$.
For small $\eps>0$ we have $0<\rho(t;\eps)<t$, and we get 
$0<t+R-\rho(t;\eps)\le t+R$.
Then it follows from \eqref{supp_t} that
$$
\|u(t)\|_E^2=\frac{1}{2}\int_{|x|\le t+R} |\pa u(t,x)|^2 dx=I_1+I_2,
$$
where
\begin{align*}
I_1
=&
\frac{1}{2}\int_{|x|\le t+R-\rho(t;\eps)} |\pa u(t,x)|^2 dx,\\
I_2
=&
\frac{1}{2}\int_{t+R-\rho(t;\eps)\le |x|\le t+R}|\pa u(t,x)|^2 dx.
\end{align*}
Note that we have $t^{-1}r\le t^{-1}(t+R)\le 1+R/2$ for $0\le r\le t+R$,
and $0<\rho(t;\eps)\le R+t-r\le R+|t-r|$ for $0\le r\le t+R-\rho(t;\eps)$.
By using the polar coordinates,
we deduce from \eqref{DecayDamping02} that
\begin{align*}
 I_1 
 \le & C\eps^2\int_0^{t+R-\rho(t;\eps)} t^{-1}(R+|t-r|)^{2\mu-2}r dr \\
 \le & 
  C\eps^2 \int_0^{t+R-\rho(t;\eps)} (R+t-r)^{2\mu-2}dr \\
 \le & 
  C\eps^2\rho(t;\eps)^{2\mu-1}
 =
 \frac{C\eps^2}{(\eps^2 \log t)^{
(1/2)-2\delta_0}},
\end{align*}
as well as 
\begin{align*}
I_2 
\le 
C \int_{t+R-\rho(t;\eps)}^{t+R} t^{-1}(\log t)^{-1} r dr
\le 
 C(\log t)^{-1} \rho(t;\eps)
 =
 \frac{C\eps^2}{(\eps^2 \log t)^{
(1/2)-2\delta_0}}.
\end{align*}
On the other hand, we get
\begin{align*}
 \|u(t)\|_E^2
  \le
  C \eps^2 \int_0^{t+R} t^{-1} \bigl(R+|t-r| \bigr)^{2\mu -2} r dr
  \le 
  C\eps^2 \int_{-\infty}^{\infty} \jb{\sigma}^{2\mu-2} d \sigma
  \le 
  C\eps^2. 
 \end{align*}
Summing up, we have  
$$
 \|u(t)\|_E^{2}
  \le 
  \frac{C \eps^{2}}
  { \bigl( 1+ \eps^2 \log t \bigr)^{
(1/2)-2\delta_0}}
 \le 
 \frac{C \eps^{2}}
  { \bigl( 1+ \eps^2 \log t \bigr)^{
(1/2)-2\delta}}
$$
for $t\ge 2$, which completes the proof of Theorem~\ref{Thm_EnergyDecay}. 
\qed

\appendix \section{Proof of Lemma~\ref{GhostEnergy}}

First we put 
$$
 \eta (t,x)=\int_{-\infty}^{|x|-t} \frac{dz}{\jb{z}^{\rho}}, 
\quad (t,x) \in \R\times \R^2.
$$
Then we can easily check that 
\begin{align}
 1
 \le 
 e^{\eta(t,x)} 
 \le 
 \exp\left(\int_{\R} \frac{dz}{\jb{z}^{\rho}}\right) 
 <
 \infty
 \label{eq_ghost}
\end{align}
and that 
$$
 (\pa_t \eta) |\pa \psi|^2 -2 (\nabla_x \eta) \cdot (\nabla_x \psi)\pa_t \psi
 =\frac{-|Z \psi|^2}{\jb{t-|x|}^\rho}.
$$
Next, as in the usual energy integral method, we compute
\begin{align*}
 &\frac{d}{dt} \left( 
  \int_{\R^2} e^{\eta} |\pa \psi|^2 dx
 \right)\\
 &=
 \int_{\R^2} 
   \bigl(e^{\eta} (\pa_t \eta)|\pa \psi|^2
   +
   2e^{\eta} \bigl\{
   (\pa_t \psi)(\pa_t^2 \psi) + (\nabla_x \psi)\cdot (\nabla_x \pa_t\psi)
  \bigr\}\bigr) dx\\
 &=
 2\int_{\R^2} e^{\eta} (\dal \psi)(\pa_t \psi) dx
 +
 \int_{\R^2} e^{\eta} 
  \bigl\{
   (\pa_t \eta)|\pa \psi|^2 -2(\nabla_x \eta)\cdot (\nabla_x \psi)\pa_t \psi
  \bigr\}
 dx
 \\
 &=
 2\int_{\R^2} e^{\eta} G \pa_t\psi  dx
 -
 \int_{\R^2} e^{\eta} \frac{|Z\psi|^2}{\jb{t-|x|}^{\rho}} dx.
\end{align*}
By the integration with respect to $t$, we have
\begin{align*}
 &\int_{\R^2} e^{\eta(t,x)} |\pa \psi(t,x)|^2 dx
 + 
 \int_0^t \int_{\R^2} 
  e^{\eta(\tau,x)} \frac{|Z\psi(\tau,x)|^2}{\jb{\tau-|x|}^{\rho}} 
 dx d\tau\\
 &= 
 \int_{\R^2} e^{\eta(0,x)} |\pa \psi(0,x)|^2 dx
 + 
 2\int_0^t \int_{\R^2} e^{\eta(\tau,x)} G(\tau,x)(\pa_t \psi)(\tau,x) dx d\tau.
\end{align*}
With the aid of \eqref{eq_ghost}, we arrive  at the desired estimate. 
\qed



\end{document}